\documentclass[11pt,reqno]{amsart} % use larger type; default would be 10pt

\usepackage[utf8]{inputenc} % set input encoding (not needed with XeLaTeX)
\usepackage{bbm}

\usepackage[margin=1in]{geometry} % to change the page dimensions
\usepackage{graphicx} % support the \includegraphics command and options
\usepackage{float} %have graphics not floating around

 \usepackage[parfill]{parskip} % Activate to begin paragraphs with an empty line rather than an indent
 
\usepackage{booktabs} % for much better looking tables
\usepackage{array} % for better arrays (eg matrices) in maths
\usepackage{paralist} % very flexible & customisable lists (eg. enumerate/itemize, etc.)
\usepackage{verbatim} % adds environment for commenting out blocks of text & for better verbatim
\usepackage{subfig} % make it possible to include more than one captioned figure/table in a single float
\usepackage{mathrsfs}
\usepackage{amssymb}
\usepackage{amsthm}
\usepackage{amsmath,amsfonts,amssymb,esint}
\usepackage{graphics,color}
\usepackage{enumerate, enumitem}
\usepackage{mathtools,centernot}
\usepackage{cases}
\usepackage{amsrefs}

\usepackage{xfrac}
\newtheorem{theorem}{Theorem}[section]
\newtheorem{lemma}[theorem]{Lemma}
\newtheorem{corollary}[theorem]{Corollary}
\newtheorem{definition}[theorem]{Definition}
\newtheorem{proposition}[theorem]{Proposition}
\newtheorem{remark}[theorem]{Remark}

\newtheorem{conjecture}[theorem]{Conjecture}

\numberwithin{equation}{section}

\newcommand{\norm}[1]{\left\|#1\right\|}
\newcommand{\abs}[1]{\left|#1\right|}

\DeclarePairedDelimiter{\ceil}{\lceil}{\rceil}

\newcommand{\T}{\ensuremath{\mathbb{T}}}
\newcommand*{\R}{\ensuremath{\mathbb{R}}}
\renewcommand*{\S}{\ensuremath{\mathcal{S}}}

\newcommand*{\N}{\ensuremath{\mathbb{N}}}
\newcommand*{\Z}{\ensuremath{\mathbb{Z}}}

\newcommand*{\Id}{\ensuremath{\mathrm{Id}}}

\newcommand*{\RR}{\ensuremath{\mathcal{R}}}

\usepackage{color, graphicx}
\usepackage{mathrsfs, dsfont}

\usepackage{hyperref}

\def\dist{\mathop{\rm dist}\nolimits}    %distance
\def\div{\mathop{\rm div}\nolimits}    %divergence
\def\supp{\mathop{\rm supp}\nolimits}    %support
\def\curl{\mathop{\rm curl}\nolimits}    %curl
\def\dim{\mathop{\rm dim}\nolimits}

\title{Dimension of the singular set of wild H\"older solutions of the incompressible Euler equations}

\author{Luigi De Rosa}
\address{EPFL SB, Station 8, 
CH-1015 Lausanne, Switzerland}
\email{luigi.derosa@epfl.ch}

\author{Silja Haffter}
\address{EPFL SB, Station 8, 
CH-1015 Lausanne, Switzerland}
\email{silja.haffter@epfl.ch}

\date{\today}
\begin{document}

\begin{abstract}
For $\beta<\frac13$, we consider $C^\beta\left(\T^3\times [0,T]\right)$ weak solutions of the incompressible Euler equations that do not conserve the kinetic energy.
We prove that for such solutions the closed and non-empty set of singular times $\mathcal{B}$ satisfies $\dim_{\mathcal{H}}(\mathcal{B})\geq \frac{2\beta}{1-\beta}$.
This lower bound on the Hausdorff dimension of the singular set in time is 
intrinsically linked to the H\"older regularity of the kinetic energy and we conjecture it to be sharp. As a first step in this direction, for every $\beta<\beta'<\frac{1}{3}$ we are able to construct, via a convex integration scheme,  non-conservative $C^\beta\left(\T^3\times [0,T]\right)$ weak solutions of the incompressible Euler system such that $\dim_{\mathcal{H}}(\mathcal{B})\leq \frac{1}{2}+\frac{1}{2}\frac{2\beta'}{1-\beta'}$. 
%The obstruction to reach the conjectured sharp Hausdorff dimension seems to be hidden in the iterative scheme and more precisely, in the time localization of the singular set. 
The structure of the wild solutions that we build allows moreover to deduce non-uniqueness of $C^\beta\left(\T^3\times [0,T]\right)$ weak solutions of the Cauchy problem for Euler from every smooth initial datum. 
\end{abstract}

\maketitle

\par
\medskip\noindent
{\textbf{MSC (2020):} 	35Q31 - 35D30  - 35A02 - 28A78.
\par
}

\section{Introduction}

We consider the incompressible Euler equations 
\begin{equation}\label{E}
\left\{\begin{array}{l}
\partial_t v+\div (v\otimes v)  +\nabla p =0\\
\div v = 0,
\end{array}\right.
\end{equation}
in the spatial periodic setting $\T^3=\R^3\setminus \Z^3$,
where $v:\T^3\times [0,T]\rightarrow \R^3$ is a vector field representing the velocity of the fluid and $p:\T^3\times [0,T]\rightarrow \R$ is the hydrodynamic pressure.

We will consider weak solutions of the system \eqref{E}, namely divergence-free vector fields $v\in L^2\left(\T^3\times [0,T];\R^3\right)$ such that 
$$
\int_0^T\int_{\T^3}\left( v\cdot \partial_t \varphi+ v\otimes v : \nabla \varphi \right)\,dx\, dt=0,
$$
for all $\varphi\in C^\infty_c\left(\T^3\times (0,T);\R^3\right)$ with $\div \varphi=0$. The pressure does not appear in the weak formulation and it can be recovered as the unique zero average solution of 
\begin{equation}\label{Lapl_p}
-\Delta p=\div \div (v\otimes v),
\end{equation}
that can be formally derived by taking the (distributional) divergence of the first line of \eqref{E}. Thus, from now on, when we talk of a vector field $v$ as weak solution of \eqref{E}, it will be understood that the pressure can be derived a posteriori by solving the elliptic equation \eqref{Lapl_p}.

In the last decade, considerable attention has been devoted to the study of H\"older continuous weak solutions of \eqref{E}, since they naturally arise in incompressible hydrodynamics models, starting from the celebrated prediction of Kolmogorov's Theory of Turbulence \cite{K41}. In this context, one of the most investigated problems is unquestionably Onsager's conjecture on the kinetic energy conservation for H\"older continuous weak solutions of \eqref{E}. Indeed, in 1949, for solutions $v\in L^\infty\left((0,T);C^\beta\left(\T^3\right)\right)\,,$  Lars Onsager predicted that anomalous dissipation of the kinetic energy 
$$
e_v(t):=\frac{1}{2}\int_{\T^3} |v(x,t)|^2\,dx
$$
may occur only in the low regularity regime $\beta<\frac{1}{3}$, while in the case $\beta>\frac13$, some rigidity of the equation prohibits the existence of such wild non-conservative weak solutions. It is worth to mention that considering weak solutions $v\in C^\beta\left(\T^3\times [0,T]\right)$ is not more restrictive with respect to the ones considered by Onsager. Indeed, in \cite{Is2013}, it has been shown that every $v\in L^\infty\left((0,T);C^\beta\left(\T^3\right)\right)$ enjoys the same $\beta-$H\"older regularity in time (see also \cite{CD18} for a different proof).

The first proof of the rigidity for $\beta>\frac{1}{3}$ was given in \cite{CET94}, while in \cite{Is2018}, for any $\beta<\frac{1}{3}$, Philip Isett proved the existence of dissipative $\beta-$H\"older continuous weak solutions of \eqref{E}. His construction is based on the iterative convex integration scheme introduced by Camillo De Lellis and László Székelyhidi in the context of incompressible fluid dynamics (see  \cite{BDLIS15}, \cite{BDLSV2019} and \cite{DS2013} for a complete introduction on the topic).  In recent years, such techniques have also been applied to other models in fluid dynamics such as the hypodissipative Navier Stokes equations \cite{CDD18, DR19}, the SQG equation \cite{BSV19}, the MHD equations \cite{BBV20} and the quasigeostrophic equation \cite{N20}. Remarkably, these convex integration techniques led to a proof of non-uniqueness of weak solutions of the Navier-Stokes system in \cite{BV2017} and of non-uniqueness for the transport equation with Sobolev vector fields in \cite{MS18, BCDL20}.
%In recent years, such techniques have also been applied to other models in fluid dynamics such as the ipodissipative Navier Stokes equations in \cite{CDD18} and \cite{DR19} to prove ill-posedness of Leray solutions, in \cite{BSV19} to prove non-uniqueness of weak solutions to the SQG equation and remarkably in \cite{BV2017} to prove non-uniqueness of weak solutions of the Navier-Stokes system.}

A natural next question to ask is how irregular those wild solutions are, or, more precisely, how small their non-empty singular set can be. In the following, we will only consider the singular set in time, that is the smallest closed set $\mathcal{B}\subseteq [0,T]$ such that $v \in C^\infty\left(\T^3 \times \mathcal{B}^c\right)\,.$ 
%we consider both the spacetime singular set, consisting of points in spacetime such that in any of their neighbourhoods the solution is non-smooth, or its projection on the time coordinate, that is the set of singular times. 
%can their singular set be, where by singular set we just mean the complement of the set in which the solution is smooth. 
This question has recently been investigated in \cite{BCV19} in the context of the Navier-Stokes equations, where the existence of wild $C^0\left([0,T];L^2\left(\T^3\right)\right)$ weak solutions whose singular set in time has Hausdorff dimension strictly less than $1$ has been established. Moreover, in the recent work \cite{CL20}, it has been shown that it is possible to construct non-conservative wild solutions of both Euler and Navier-Stokes equations whose singular set of times has arbitrarily small Hausdorff dimension if one requires only some low $L^p$ integrability in time $1\leq p<2\,.$ 
%(thus loosing the uniform regularity). 
Specifically, in the context of the Euler equations, these solutions belongs to $L^{\sfrac{3}{2}-}\left([0,T]; C^{\sfrac{1}{3}}\left(\T^3\right)\right) \cap L^1\left([0, T]; C^{1-}\left(\T^3\right)\right)$ and do not possess a uniform in time regularity.

The question on the size of the singular set in time of wild $C^\beta\left(\T^3\times [0,T]\right)$
%$\beta-$H\"older continuous (uniformly in time)
weak solutions of \eqref{E} as been raised in \cite{CL20} and has not yet been investigated.
In this work, we address this issue by studying the structure of the non-conservative weak solutions of Euler constructed in \cite{Is2018} and \cite{BDLSV2019}. We first prove that the singular set in time of such solutions cannot  be arbitrarily small. More precisely, we have the following
\begin{theorem}\label{t:lowerbound}
Let $0< \beta < \frac{1}{3}$ and $v\in C^\beta\left(\T^3\times[0,T]\right)$ be a non-conservative weak solution of \eqref{E}. If $\mathcal{B}\subseteq [0,T]$ is a closed set such that $v\in C^\infty\left(\T^3\times \mathcal{B}^c\right)\,,$ then
$\mathcal{H}^{\frac{2\beta}{1-\beta}}(\mathcal{B})>0\,.$ In particular, we have
$$
\dim_{\mathcal{H}}(\mathcal{B})\geq \frac{2\beta}{1-\beta}.
$$
\end{theorem}
The previous result is intrinsically related to the H\"older continuity of kinetic energy of the corresponding class of solutions. Indeed, a remarkable property of $\beta-$H\"older continuous weak solutions of \eqref{E} is that the corresponding kinetic energy $e_v$ enjoys the following peculiar H\"older regularity
\begin{equation}\label{e_holder}
|e_v(t)-e_v(s)|\leq C |t-s|^{\frac{2\beta}{1-\beta}},
\end{equation}  
which can also be viewed as a different proof of the energy conservation in the case $\beta>\frac13$ (since in this case $\frac{2\beta}{1-\beta}>1$). Since $e_v$ is moreover constant on $\mathcal{B}^c$, but not on $[0,T]$, Theorem \ref{t:lowerbound} quantifies how big $\mathcal{B}$ has to be in order to allow the energy $e_v$ to grow, in a $C^{2\beta/(1-\beta)}$ fashion, between its different values. In this way, Theorem \ref{t:lowerbound} is a consequence of a general property of non-constant H\"older continuous functions that increase only on a set of given Hausdorff dimension (see Lemma \ref{holder_hausdorff} below).

Property \eqref{e_holder} has been first observed in \cite{Is2013} for any $\beta\leq \frac{1}{3}$ and then generalized in \cite{CD18} for any value of $\beta$, in the slightly more general class of Besov regularity.
We also remark that recently in \cite{DT20}, property \eqref{e_holder} has been shown to be sharp in the Baire category sense, which was previously conjectured in \cite{IO16}.
Motivated by the sharpness of the energy regularity and its connection with the size of the set of singular times of non-conservative solutions, we make the following
\begin{conjecture}\label{conject}
For every $\beta < \frac{1}{3}$, there exists a non-conservative weak solution $v\in C^\beta\left(\T^3\times [0,T]\right)$ of \eqref{E} and a closed set $\mathcal{B}\subset [0,T]$, such that $v \in C^\infty\left(\T^3\times \mathcal{B}^c\right)$ and $\dim_{\mathcal{H}}(\mathcal{B})=\frac{2\beta}{1-\beta}$.
\end{conjecture}

 Observe that according to Theorem \ref{t:lowerbound} such a solution necessarily satisifes $\mathcal{H}^{\frac{2\beta}{1-\beta}}(\mathcal{B})>0\,.$  In this note, we make a first step towards the conjecture. More precisely, using the convex integration scheme of \cite{BDLSV2019} together with the time localization introduced in \cite{BCV19}, we prove the following
\begin{theorem}\label{t:main}
Let $0\leq \beta<\beta'<\frac{1}{3}$ and let $v_1,v_2\in C^\infty\left(\T^3\times [0,T]\right)$ be two smooth solutions of \eqref{E} such that $\int_{\T^3} v_1(x,t)\,dx=\int_{\T^3} v_2(x,t)\,dx$, for all $t\in [0,T]$. There exists $v\in C^\beta\left(\T^3\times [0,T]\right)$ which weakly solves \eqref{E} such that the following holds
\begin{itemize}
\item[(i)] $v\big|_{\left[0,{\sfrac{T}{3}}\right]}\equiv v_1$ and $v\big|_{\left[{\sfrac{2T}{3}}, T\right]}\equiv v_2$;
\item[(ii)] there exists a closed set  $\mathcal{B}\subset [0,T]$ such that $v\in C^\infty\left(\T^3\times \mathcal{B}^c\right)$ and $\dim_{\mathcal{H}}(\mathcal{B})\leq \frac{1}{2}+\frac{1}{2}\frac{2\beta'}{1-\beta'}$.
\end{itemize}
\end{theorem}
The previous result is in the spirit of the result \cite[Theorem 1.1]{BCV19} for the Navier-Stokes equations: as the former, it gives on one hand a strong non-uniqueness result for the Cauchy problem of the Euler equations.  Indeed, for any smooth initial datum $\overline  v\in C^\infty\left(\T^3\right)$ one can choose $v_1\in C^\infty\left(\T^3\times [0,T]\right)$ as the smooth solution such that $v_1(0,\cdot)=\overline  v$, where $T>0$ is its maximal time of existence, and as $v_2$ any stationary smooth solution which differs from $\overline  v$. This clearly shows that for every $\beta< \frac{1}{3}\,,$ $C^\beta\left( \T^3\times [0,T]\right)$ weak solutions are non-unique for every smooth initial datum. We remark that, in view of the weak-strong uniqueness result from \cite{BDS11}, our solutions can not be admissible, in the sense that they do not verify $e_v(t)\leq e_v(0)$ for every $t\in [0,T]$. For a non-uniqueness result on such solutions we refer to \cite{DS17, DRS20}, in which the $L^2-$density of wild initial data has been recently established up to the $\frac13 -$Onsager's critical threshold. On the other hand, Theorem \ref{t:main} builds solutions that are smooth outside a compact set of quantifiable Hausdorff dimension. The hypothesis on the spatial averages of the two smooth solution is just a standard compatibility condition, since every continuous solution of \eqref{E} preserves its mean on the torus.

%\textcolor{blue}{
%Even if our main theorem proves non-uniqueness of weak solutions starting from every smooth initial datum,   we remark that, in view of the weak-strong uniqueness result from \cite{BDS11}, our solutions can not be admissible, in the sense that they do not verify $e_v(t)\leq e_v(0)$ for every $t\in [0,T]$. For a non-uniqueness result on such solutions we refer to \cite{DS17, DRS20}, in which the $L^2-$density of wild initial data has been recently established up to the $\frac13 -$Onsager's critical threshold.}

The loss given by the gap $\beta'>\beta$ is typical of such iterative schemes as already observed in \cite[Theorem 1.1]{DT20}, while the gap between the Hausdorff dimension achieved in $(ii)$ and the one of Conjecture \ref{conject} is an outcome of the implementation of the time localization of \cite{BCV19} in the scheme of \cite{BDLSV2019}, that we believe could be improved. We postpone the technical discussion of this issue to Section \ref{sec_gap}.

\subsection*{Aknowledgements}
The authors aknowledge the support of the SNF Grant $200021\_182565$.
The authors would also like to thank Maria Colombo for her interest in this problem and the useful discussions about it.

\section{Outline of the proof and main iterative scheme}

In order to construct H\"older continuous solutions of Euler we will base our construction on the convex integration scheme proposed in \cite{BDLSV2019}. However, there will be two main differences: at first,  since the main goal of our Theorem \ref{t:main} is to ensure that the constructed solution is smooth in a large set of times, we need to introduce a time localization of the glued Reynolds stress as well as of the perturbation. This will be done by adapting the idea that has been introduced in \cite{BCV19} in the context of $L^p$-based convex integration for the incompressible  Navier-Stokes equations. Second, since our purpose in not to prescribe a given energy profile $e=e(t)$, we will avoid all the technicalities coming from the energy iterations. We remark that, even if an energy profile will not be prescribed, the failure of energy conservation will still be a consequence of Theorem \ref{t:main}, since we can glue two solutions $v_1$ and $v_2$ whose kinetic energy differs. We begin with the simple proof of Theorem \ref{t:lowerbound} and then we move on to the description of the main iteration.

\subsection{Proof of Theorem \ref{t:lowerbound}} 
%As already pointed out, our Theorem \ref{t:lowerbound} is a consequence of a general measure theory result on piecewise constant H\"older continuous functions. 
The following lemma asserts that a $\theta-$H\"older continuous function, defined on a $1$-dimensional domain, cannot increase \textit{only} on a null set of the $\theta$-dimensional Hausdorff measure. Since we could not find a reference for it, we give a detailed proof.

\begin{lemma}\label{holder_hausdorff}
Let $e\in C^\theta\left([0,T]\right)$ for some $\theta\in (0,1)$ and let $\mathcal{B}\subset [0,T]$ be a closed set with $\mathcal{H}^\theta(\mathcal{B})=0$. If $\frac{d}{dt}e=0$ on $\mathcal{B}^c$, then $e(t)=e(0)$ for all $t\in [0,T]$.
\end{lemma}

\begin{proof}
Since $\mathcal{H}^\theta(\mathcal{B})=0$, for every $\varepsilon>0$, there exists a family of open balls $\left\{B_{r_i}(t_i) \right\}_i$, such that  $\mathcal{B}\subset \bigcup_i B_{r_i}(t_i)$ and 
\begin{equation}\label{small}
\sum_{i} r_i^\theta<\varepsilon.
\end{equation}
Moreover, since $\frac{d}{dt}e\Big|_{\mathcal{B}^c}=0$, then the function $e$ can not increase (nor decrease) on $\left(\bigcup_i B_{r_i}(t_i)\right)^c$. This implies that $\forall t\geq 0$, it holds
$$
|e(t)-e(0)|\leq \sum_i |e(t_i-r_i)-e(t_i+r_i)|,
$$
which together with the $\theta$-H\"older continuity of $e$ and \eqref{small}, allows us to conclude 
$$
|e(t)-e(0)|\leq C\sum_i r_i^\theta <C\varepsilon.
$$
The claim then follows since $\varepsilon>0$ was arbitrary.
\end{proof}
To prove Theorem \ref{t:lowerbound}, just notice that the kinetic energy $e_v$ of any solution $v\in C^\beta\left(\T^3\times [0,T]\right)\cap C^\infty\left(\T^3\times \mathcal{B}^c\right)$ always satisfies \eqref{e_holder}, and moreover, by the standard energy conservation for smooth solutions, we also get
$$
\frac{d}{dt}e\bigg|_{\mathcal{B}^c}=0.
$$
Then Lemma \ref{holder_hausdorff}, together with the assumption that $v$ is non-conservative, implies 
$\mathcal{H}^\frac{2\beta}{1-\beta}(\mathcal{B})>0$ and hence in particular
the desired lower bound $\dim_{\mathcal{H}}\mathcal{B}\geq \frac{2\beta}{1-\beta}$.

\subsection{Inductive proposition}\label{s:inductive_prop} For any index $q\in \mathbb{N}$ we will construct a smooth solution $(v_q,R_q)$ of the Euler Reynolds system on $\T^3 \times [0, T]$
\begin{equation}\label{ER}
\left\{\begin{array}{l}
\partial_t v_q+\div (v_q\otimes v_q)  +\nabla p_q =\div R_q\\
\div v_q = 0,
\end{array}\right.
\end{equation}
where $R_q$ is a symmetric matrix. The pressure $p_q$ will consequently be the unique zero average solution of 
\begin{equation}\label{e:pressure}
-\Delta p_q=\div \div (v_q \otimes v_q- R_q).
\end{equation}
For any integer $q\,,$ we define a frequency parameter $\lambda_q$ and an amplitude parameter $\delta_q$ by 
\begin{align*}
\lambda_q&=2\pi \lceil a^{b^q}\rceil,\\
\delta_q&=\lambda_q^{-2\beta},
\end{align*}
where $0<\beta<\frac{1}{3}$ is the regularity exponent of Theorem \ref{t:main}, $b>1$ is a number that is close to $1$ and $a\gg1$ is a large enough parameter that will be chosen at the end (depending on all the other parameters). We also introduce the parameter 
\begin{equation}\label{gamma:bound1}
\gamma \in (0, (b-1)(1-\beta)),
\end{equation}
which will be the key parameter to measure the smallness of the singular set of the solution $v$ that we construct, as well as the parameter $\alpha>0$, which will be chosen sufficiently small (depending on $\beta\,,$ $b$ and $\gamma$), together with the universal geometric constant $M>0$ that will be defined later in the construction.

At step $q$, we will assume the following iterative estimates on the couple $(v_q,R_q)$
\begin{align}
\|R_q\|_0&\leq \delta_{q+1}\lambda_q^{-\gamma-3\alpha} \label{R_q:C0_est}\,,\\
\|v_q\|_1&\leq M \delta_q^{\sfrac{1}{2}}\lambda_q\label{v_q:C1_est}\,,\\
\|v_q\|_0&\leq 1-\delta_q^{\sfrac{1}{2}} \label{vq:C0_est}\,.
\end{align}
%where $\gamma >0$ will be the key parameter to measure the smallness of the singular set of the solution $v$ given in Theorem \ref{t:main}, $\alpha>0$ is a parameter that will be chosen sufficiently small (depending on $\beta\,,$ $b$ and $\gamma$) and $M>0$ is a universal geometric constant that will be defined later in the construction.
Here the H\"older norms will always only measure the spatial regularity; in other words, we take the supremum in time of the corresponding spatial H\"older norm (see Appendix \ref{s:hoelder}).

We will also inductively assume that the vector field $v_q$ is an exact solution of \eqref{E} for a large set of times, or analogously that the support of the Reynolds stress $R_q$ is contained in a finite union of thiny time intervals. To this aim, we follow the scheme of time localizations introduced in \cite{BCV19} and, for $q \geq 1$, we introduce the following two parameters
\begin{align*}
\theta_q&=\frac{1}{\delta^{\sfrac{1}{2}}_{q-1}\lambda^{1+3\alpha}_{q-1}}\,,\\
 \tau_q&=\lambda_{q-1}^{-\gamma}\theta_q.
\end{align*}
For the special case $q=0\,,$ we set $\tau_0=\sfrac{T}{15}$, while for $\theta_0$ we don't need to assign any value.
%\tau_0 needs to be 1/15 because we need 5 \tau_0=1/3.
Observe that for every $q \geq 1$, we have as a consequence on the bounds on $\gamma$ in \eqref{gamma:bound1} that
\begin{equation}\label{e:timeparameters}
\theta_{q+1}\ll \tau_q \ll \theta_q \ll 1 \,.
\end{equation}

In order to ensure $(ii)$ in Theorem \ref{t:main}, we split the time interval $[0,T]$ at step $q \geq 0$ into a closed good set $\mathcal{G}_q$ and an open bad set $\mathcal{B}_q$ such that 
$[0, T] =  \mathcal{G}_q \cup \mathcal{B}_q $ and $\mathcal{G}_q \cap \mathcal{B}_q= \emptyset \,.$ The Reynolds stress will be supported strictly inside the bad set and hence, on the good set $v_q$ will be a smooth solution of \eqref{E}. More precisely, we will inductively construct the sets $\mathcal{G}_q$ and $\mathcal{B}_q$ with the following properties:
\begin{enumerate}[label=(\roman*)]
\item\label{Bq:i} $\mathcal{G}_0:= [0,\sfrac{T}{3}] \cup [\sfrac{2T}{3}, T] \,,$
\item\label{Bq:ii} $\mathcal{G}_{q-1} \subset \mathcal{G}_q$ for all $q \geq 1\,,$
\item\label{Bq:iii} $\mathcal{B}_q$ is a finite union of disjoint open intervals of length $5 \tau_q\,,$
\item\label{Bq:iv} the size of $\mathcal{B}_q$ is shrinking in $q$ according to the rate
\begin{equation}\label{B_q:size}
\lvert \mathcal{B}_q \rvert \leq 10 \frac{\tau_q}{\theta_q} \lvert \mathcal{B}_{q-1} \rvert \quad \forall
 q\geq1,
\end{equation}
\item\label{Bq:v}  if $t \in \mathcal{G}_{q'}$ for some $q' <q$, then $v_q(t)= v_{q'}(t)\,,$
\item\label{Bq:vi} defining the ``real" bad set  $\widehat{\mathcal{B}}_q :=\{ t \in [0, T]:\, \dist(t, \mathcal{G}_q) > \tau_q \}$, we have that the Reynolds stress $R_q$ is supported inside $\widehat{\mathcal{B}}_q$, or in other words
\begin{equation}\label{R_q:zeroonbadset}
R_q(t) \equiv 0 \quad \text{ for all } t \in \widehat{\mathcal{B}}_q^c \,,
\end{equation} 
\item\label{Bq:vii} on the complement of the real bad set, $v_q$ (that from (vi) is a smooth solution of Euler) satisfies the better estimate 
\begin{equation}\label{v_q:betterestimate}
\norm{v_q(t)}_{N+1} \lesssim   \delta_{q-1}^{\sfrac{1}{2}} \lambda_{q-1}\ell_{q-1}^{-N} \quad \text{ for all } t \in \widehat{\mathcal{B}}_q^c \,,
\end{equation} 
for all $q\geq 1$, where $\ell_{q-1}$ is the mollification parameter, as introduced in \eqref{e:l}. Here, the symbol $\lesssim$ means that the constant in the inequality is allowed to depend on $N\,,$ but not on any of the parameters and, in particular, not on $q$.
\end{enumerate}

The following iterative proposition is the cornerstone of the proof of Theorem \ref{t:main}.
\begin{proposition}[Iterative Proposition]\label{p:iterativeprop}There exists a universal constant $M>0$ such that the following holds. Fix $0<\beta<\frac{1}{3}\,,$ $1< b < \frac{1-\beta}{2\beta}$ and 
\begin{equation}\label{gamma:bound2}
0< \gamma < \frac{(b-1)(1-\beta - 2\beta b)}{b+1}\,.
\end{equation} 
Then, there exists $\alpha_0=\alpha_0(\beta, b, \gamma)>0$ such that for every $0<\alpha< \alpha_0\,,$ there exists $a_0=a_0(\beta, b, \gamma, \alpha, M)$ such that for every $a \geq a_0$ the following holds.

Given a smooth couple $(v_q, R_q)$ solving \eqref{ER} on $\T^3 \times [0,T]$ with the estimates \eqref{R_q:C0_est}--\eqref{vq:C0_est} and a set $\mathcal{B}_q\subset [0,T]$ satisfying the properties {\normalfont  \ref{Bq:i}--\ref{Bq:vii}} above, there exists a smooth solution $(v_{q+1}, R_{q+1})$  to \eqref{ER} on $\T^3 \times [0,T]$ and a set $\mathcal{B}_{q+1}\subset [0,T]$ satisfying both the estimates  \eqref{R_q:C0_est}--\eqref{vq:C0_est} and the properties {\normalfont \ref{Bq:i}--\ref{Bq:vii}} with $q$ replaced by $q+1\,.$ Moreover, we have 
\begin{equation}\label{v_q+1-v_q}
\norm{v_{q+1}- v_q}_0 + \lambda_{q+1}^{-1} \norm{v_{q+1}- v_q}_1 \leq M \delta_{q+1}^{\sfrac{1}{2}} \,.
\end{equation}
\end{proposition}

%It may seem inconsistent to settle Proposition \ref{p:iterativeprop} on an arbitrary time interval $[0, T]$ whereas the main Theorem \ref{t:main} is stated on $[0,1]$. However, in order to deduce Theorem \ref{t:main} from Proposition \ref{p:iterativeprop} by means of an iterative construction, we need to perform a first rescaling, depending on the size of the smooth solutions $v_{1}$ and $v_{2}$, to reach the required smallness of first step of the iteration; the iterative procedure will then be carried out on the rescaled time interval $[0,T]$ and only at the very end, we will undo the rescaling to come back to interval $[0,1]$ (see proof of Theorem \ref{t:main}).  

The proof of the main inductive proposition will occupy almost all of the remaining paper; we give a sketch of the different steps in the proof in the sections \ref{s:glueing} and \ref{s:convexintegration}. Before doing so, we show how the size of the singular set in time is linked to the choice of the parameters and how the iterative proposition implies Theorem \ref{t:main}.

\subsection{Size of the singular set in time}\label{s:sizeofsingularset} From \eqref{R_q:zeroonbadset}, it follows that $v_q$ is a smooth solution to Euler on $\T^3 \times \mathcal{G}_q\,.$ Moreover, the estimate \eqref{v_q+1-v_q}, together with the fact that $R_q\rightarrow 0$ uniformly from \eqref{R_q:C0_est},  will ensure that $v_q$ converges strongly in $C^0\left(\T^3 \times [0,T]\right)$ to a weak solution $v$ of \eqref{E} (see proof of Theorem \ref{t:main}). Property \ref{Bq:v} guarantees that $v = v_q$ on $\mathcal{G}_q$ and hence the limit solution will be smooth in $\T^3 \times \mathcal{G}_q\,.$ Since this holds for every $q\geq 0$, we deduce that there exists a closed set $\mathcal{B}\subset [0,T]$, of zero Lebesgue measure, such that  $v \in C^\infty\left(\T^3 \times \mathcal{B}^c\right)$ and moreover,
\begin{equation}\label{B}
\mathcal{B}\subset \bigcap_{q \geq 0} \mathcal{B}_q \,.
\end{equation}
The shrinking rate \eqref{B_q:size} allows us to estimate the Hausdorff (in fact the box-counting) dimension of the right-hand side. Indeed, using also the definition of the parameters $\tau_q$, $\theta_q$ and $\lambda_q$, we have
\begin{equation}
\lvert \mathcal{B}_q \rvert \leq \lvert \mathcal B_0 \rvert \prod_{q'=1}^q 10\frac{\tau_{q'}}{\theta_{q'}}  = \frac{10^q T}{3}\prod_{q'=1}^q \lambda_{q'-1}^{-\gamma} \leq (40\pi)^q T a^{-\gamma \left(\frac{b^{q}-1}{b-1}\right)}\,.
\end{equation}
Since by (iii) every $\mathcal{B}_q$ is made of disjoint intervals of length $5 \tau_q$, this implies that for every $q\geq 0$, the set $\mathcal{B}_q$ (and hence $\mathcal B$) is covered by at most 
\begin{equation}
 (40\pi)^q T a^{-\gamma \left(\frac{b^{q}-1}{b-1}\right)} (5 \tau_q)^{-1}
\end{equation}
of such intervals. Since $\tau_q \rightarrow 0$ as $q \to \infty$, this shows that the box-counting dimension (and hence the Hausdorff dimension) of $\mathcal{B}$ is bounded by
\begin{align}\label{dim_est_B}
\dim_{b} (\mathcal{B}) &\leq \lim_{q \to \infty} -\frac{\log\Big( (40\pi)^q T a^{-\gamma \left(\frac{b^{q}-1}{b-1}\right)} (5 \tau_q)^{-1}\Big)}{\log(5\tau_q)} \nonumber \\
&\leq  1+\lim_{q \to \infty} \frac{\gamma (b^q-1) \log a}{(b-1) \log \tau_q} \nonumber\\
&=  1- \frac{\gamma b}{(b-1)(1-\beta +3\alpha +\gamma)}\,,
\end{align}
where in the last equality we used that by definition $\tau_q= \lambda_{q-1}^{-(1-\beta+\gamma+ 3 \alpha)}\,.$ Observe that for $\gamma$ in the range \eqref{gamma:bound1}, this dimension estimate makes sense for $\alpha$ small enough, that is $\dim_{b} (\mathcal{B}) \in (0, 1)\,.$

\subsection{Proof of Theorem \ref{t:main}} We fix $0<\beta<\beta' <{\sfrac{1}{3}}$ and we define the auxiliary parameter
\begin{equation*}
\beta'':=\frac{\beta+\beta'}{2}\,.
\end{equation*}
Let $1<b< \frac{1-\beta''}{2\beta''}$ and let $\gamma \in \left(0, \frac{(b-1)(1-\beta''-2\beta'' b)}{b+1}\right)$ yet to be chosen. We will apply Proposition \ref{p:iterativeprop} with the parameters $(\beta'', b, \gamma)$ and we therefore fix admissible parameters $\alpha \in (0, \alpha_0)$ and $a \geq a_0$, where $\alpha_0$ and $a_0$ are given by the proposition.

Let $v_1, v_2 \in C^\infty\left(\T^3 \times [0, T]\right)$ be two smooth solutions  of \eqref{E} with the same spatial average. We construct the desired gluing $v$ with an inductive procedure. To this aim, let $\eta:[0,T] \rightarrow [0,1]$ be a smooth cutoff function such that $\eta \equiv 1 $ on $[0, {\sfrac{2T}{5}}]$ and $\eta \equiv 0$ on $[{\sfrac{3T}{5}}, T]\,.$ Consequently, we define the starting velocity $v_0$ as
$$  
v_0 (x,t):= \eta(t) v_1(x,t) + (1-\eta(t)) v_2(x,t).
$$
In order to define the Reynolds tensor $R_0$, we recall the inverse divergence operator $\mathcal R$ from \cite{BDLSV2019}, that is defined as
\begin{equation}
\label{e:R:def}
\begin{split}
({\mathcal R} f)^{ij} &= {\mathcal R}^{ijk} f^k, \\
{\mathcal R}^{ijk} &:= - \frac 12 \Delta^{-2} \partial_i \partial_j \partial_k - \frac 12 \Delta^{-1} \partial_k \delta_{ij} +  \Delta^{-1} \partial_i \delta_{jk} +  \Delta^{-1} \partial_j \delta_{ik},
\end{split}
\end{equation}
when acting on zero average vectors $f$ and has the property that $\mathcal R f$ is symmetric and $\div ( {\mathcal R}  f) = f$. Thus, we define
\begin{align*}
R_0= \partial_t \eta \mathcal{R}(v_1-v_2) - \eta(1-\eta) (v_1-v_2)\otimes(v_1-v_2)\,.
\end{align*}
Note that, the first term in the definition of $R_0$ is well defined since $\int_{\T^3}(v_1-v_2)\,dx=0$ by assumption.

The smooth couple $(v_0, R_0)$ solves \eqref{ER} however, it does not verify the bounds \eqref{R_q:C0_est}--\eqref{vq:C0_est} at $q=0\,.$ To bypass this problem, we exploit the invariance of the Euler equations under the rescaling 
\begin{equation}\label{e:rescaling}
\left(v_0, R_0\right) \rightarrow \left(v_0^\varepsilon(x,t)= \varepsilon v_0(x,\varepsilon t), \,R_0^\varepsilon(x,t)= \varepsilon^2 R_0(x, \varepsilon t)\right)\,.
\end{equation}
Observe that $(v_0^\varepsilon, R_0^\varepsilon)$ is a smooth solution of \eqref{ER} on $\T^3 \times \left[0, \varepsilon^{-1} T\right]$, with the properties that
\begin{align}\label{e:v0property}
v_0^\varepsilon \big|_{\left[0, \sfrac{\varepsilon^{-1} T}{3}\right]}\equiv v_1^\varepsilon \quad \text{ and } \quad v_0^\varepsilon \big|_{\left[ \sfrac{ \varepsilon^{-1}2T}{3}, \varepsilon^{-1} T\right]}\equiv v_2^\varepsilon \,, \\
\norm{ R_0^\varepsilon}_0 = \varepsilon^2 \norm{ R_0}_0, \, \quad \norm{v_0^\varepsilon}_0= \varepsilon \norm{ v_0}_0 \quad  \text{ and } \norm{ v_0^\varepsilon}_1 = \varepsilon \|v_0\|_1\,. \nonumber
\end{align}
This allows to choose $\varepsilon$ small enough, depending on all the previous parameters and  additionally also on $T, v_1$ and $v_2$, in order to satisfy \eqref{R_q:C0_est}--\eqref{vq:C0_est} at $q=0\,.$ To be precise, we choose
\begin{equation}
\varepsilon= \min \bigg\{ \bigg(\frac{\delta_1 \lambda_0^{-(\gamma +3 \alpha)}}{\norm{R_0}_0} \bigg)^{\sfrac{1}{2}}, \frac{M \delta_0^{\sfrac{1}{2}}\lambda_0}{\norm{v_0}_0}, \frac{1-\delta_0^{\sfrac{1}{2}}}{\norm{v_0}_1}  \bigg\} .
\end{equation}

With this choice of $\varepsilon$, $(v_0^\varepsilon, R_0^\varepsilon)$ satisfies the estimates \eqref{R_q:C0_est}--\eqref{vq:C0_est} as well as the properties \ref{Bq:i}--\ref{Bq:vii} for $q=0$ (where for \ref{Bq:vii} the constant in the inequality \eqref{v_q:betterestimate} depends on $\varepsilon, \norm{v_1}_N, \norm{v_2}_N$). We then apply inductively Proposition \ref{p:iterativeprop}. We start from $q=0$ with the couple $(v_0^\varepsilon, R_0^\varepsilon)$ solving \eqref{ER} on  $\T^3 \times \left[0, \varepsilon^{-1} T\right]$ and the bad set $\mathcal{B}_0= \left(\sfrac{\varepsilon^{-1}T}{3}, \sfrac{\varepsilon^{-1}2T}{3}\right)$. In this way, we construct a sequence of smooth solutions $\big\{ (v_q^\varepsilon, R_q^\varepsilon) \big\}_{q \geq 0}$ to \eqref{ER} on  $\T^3 \times \left[0, \varepsilon^{-1} T\right]$, with estimates \eqref{R_q:C0_est}--\eqref{vq:C0_est} and \eqref{v_q+1-v_q}, and with the corresponding bad set $\mathcal{B}_q$ obeying \ref{Bq:i}--\ref{Bq:vii}. The bound \eqref{v_q+1-v_q} implies, together with the interpolation estimate \eqref{e:Holderinterpolation2}, that
\begin{equation*}
\sum_{q=0}^\infty \norm{ v_{q+1}-v_q}_{\beta} \lesssim \sum_{q=0}^\infty \norm{v_{q+1}-v_q}_1^{\beta} \norm{v_{q+1}-v_q}_0^{1-\beta} \lesssim \sum_{q=0}^\infty \lambda_{q+1}^{\beta -\beta''} \lesssim 1 \,.
\end{equation*}
Hence there a exists a strong limit
\begin{equation}
w= \lim_{q \to \infty} v_q^\varepsilon \in C^0\left(\left[0, \varepsilon^{-1} T\right], C^\beta \left(\T^3\right)\right) \,.
\end{equation}
By \eqref{R_q:C0_est} we have that $R_q^\varepsilon\rightarrow 0$ uniformly, which implies that the limit $w$ solves \eqref{E}. From \cite{CD18}, we also recover the regularity in time and we deduce that in fact $w \in C^\beta\left(\T^3 \times \left[0, \varepsilon^{-1} T\right]\right)$. 

Combining the properties \ref{Bq:i}, \ref{Bq:ii} and \ref{Bq:v} of the bad set and recalling the structure of $v_0^\varepsilon$ from \eqref{e:v0property}, we deduce that 
\begin{equation}
w \big|_{\left[0, \sfrac{\varepsilon^{-1} T}{3}\right]}\equiv v_1^\varepsilon \quad \text{ and } \quad w\big|_{\left[ \sfrac{ \varepsilon^{-1}2T}{3}, \varepsilon^{-1} T\right]}\equiv v_2^\varepsilon . \\
\end{equation}
Moreover, as proven in Section \ref{s:sizeofsingularset}, there exists a closed set $\mathcal{C}\subset \bigcap_{q \geq 0} \mathcal{B}_q \subset \left[0, \varepsilon^{-1}T\right] $ such that $w \in C^\infty\left(\T^3 \times \mathcal{C}^c\right)$ and 
\begin{equation*}
\dim_b(\mathcal{C}) \leq 1- \frac{ \gamma \beta''}{(b-1)(1-\beta'' + 3\alpha+ \gamma)}\,.
\end{equation*} 
We now come to the choice of the parameters $b, \gamma$ and $\alpha \,.$ It is easy to observe that the infimum of the above dimension bound is reached in the limit as $\alpha \downarrow 0$, $\gamma \uparrow \frac{ (b-1)(1-\beta''-2\beta'' b)}{b+1} $ and $b \downarrow 1\,.$ More precisely, we have
\begin{equation}\label{e:infdim}
\inf_{ b \in \left(1, \frac{1-\beta''}{2\beta''}\right)}  \, \left \{ \inf_{ \gamma \in \left(0, \frac{(b-1)(1-\beta''-2b \beta''))}{b+1}\right)} \left \{ \inf_{\alpha \in (0, \alpha_0)} \, \left \{  1- \frac{ \gamma \beta''}{(b-1)(1-\beta'' + 3\alpha+ \gamma)} \right \} \right \} \right \}= \frac{1}{2} + \frac{1}{2} \frac{2\beta''}{1-\beta''}\,.
\end{equation}
Since by choice of $\beta''>\beta'$, the right-hand side of \eqref{e:infdim} is strictly smaller than the desired dimension bound $\frac{1}{2}+ \frac{1}{2} \frac{2\beta'}{1-\beta'}\,,$ we can first choose $b$ sufficiently close to 1  (depending on $\beta''$), then $\gamma$ sufficiently close to $\frac{(1-b)(1-\beta''-2\beta''b)}{b+1}$ (depending on $\beta''$ and $b$) and finally $\alpha$ sufficiently close to $0$, such that 
\begin{equation*}
\dim_b(\mathcal{C}) \leq 1- \frac{ \gamma \beta''}{(b-1)(1-\beta'' + 3\alpha+ \gamma)}\leq \frac{1}{2}+ \frac{1}{2} \frac{2\beta'}{1-\beta'}\,.
\end{equation*}

Finally, we rescale back and set $v(x,t)= \varepsilon^{-1} w\left( x, \varepsilon^{-1} t\right) $ to obtain a weak solution $v \in C^\beta\left(\T^3 \times [0,T]\right)$ which is a gluing of $v_1$ and $v_2$ (in the sense of Theorem \ref{t:main}) and which is smooth in $\T^3 \times \mathcal{B}^c$, where 
\begin{equation}
\mathcal{B}= \left\{ \varepsilon^{-1} t: t \in \mathcal{C} \right\}.
\end{equation}
By scale-invariance, $\mathcal{B}$ obeys the same desired Hausdorff dimension bound as $\mathcal{C}\,.$

\qed

\subsection{Gluing and localization step}\label{s:glueing} As a first step in the proof of Proposition \ref{p:iterativeprop}, we construct from the couple $(v_q, R_q)$ and the set $\mathcal{B}_q$ a new couple $\left(\overline  v_q, \overline  R_q\right)$ solving \eqref{ER} as well as a set $\mathcal{B}_{q+1}$ satisfying (i)--(vii), with $q$ replaced by $q+1$. Whereas $\overline v_q$ will enjoy roughly the same estimates as $v_q$,  the new Reynolds stress $\overline R_q$ will already be localized (in time) in a subset of $\widehat{\mathcal{B}}_{q+1}$, that is in disjoint intervals of length $ \tau_{q+1}$. The price of this localization in time will be worsened estimates on $\overline R_q$ with respect to $R_q\,, $ proportional to shrinking rate \eqref{B_q:size}.

Following the construction of \cite{Is2018}, $\overline v_q$ will be a gluing of exact solutions of the Euler equations. In order to produce those solutions, we first mollify $v_q$ in space at length scale $\ell$, as it is typical in convex integration schemes for the Euler equations to avoid the loss of derivative problem.\footnote{In our context, this means more precisely that we are not able to prescribe estimates on a finite number of $N$ derivatives (independent on the step $q$) in the inductive scheme, nor estimates on derivatives of any order with constants independent on the step $q$.} To this end, let $\varphi$ be standard radial mollification kernel in space which we rescale with some parameter  $\ell_q$, that is $\varphi_{\ell_q}(x)= \ell_q^{-3} \varphi\left(\frac{x}{\ell_q}\right)\,.$ For any $q\geq 1$, we choose the mollification parameter to be
\begin{equation}\label{e:l}
\ell_{q}= \frac{\delta_{q+1}^{\sfrac{1}{2}}}{\delta_q^{\sfrac{1}{2}} \lambda_q^{1+\sfrac{\gamma}{2} +\sfrac{3\alpha}{2}}}\,.
\end{equation}
Observe that  in view of \eqref{gamma:bound1}, $\ell_{q}$ enjoys, for $\alpha$ small enough\footnote{The upper bound $\ell_q\leq \lambda_q^{-1}$ holds for every $\alpha>0$, while for the lower bound $\ell_q\geq \lambda_q^{-\sfrac{3}{2}}$ it suffices to require $\alpha < \beta$. }, the elementary bounds 
\begin{equation}\label{l:bound}
 \lambda_{q}^{-\sfrac{3}{2}}< \ell_{q} < \lambda_q^{-1} \,.
\end{equation}
In what follows, we will usually drop the subscript $q$ unless there is ambiguity about the step.

We  define the mollified functions
\begin{align*}
v_{\ell}&= v_{q} \ast \varphi_{\ell}, \\
R_{\ell}&= R_{q} \ast \varphi_{\ell} + (v_{q} \otimes v_{q}) \ast \varphi_{\ell}- v_{\ell} \otimes v_{\ell},\\
p_{\ell}&=p_{q} \ast \varphi_{\ell}.
\end{align*}
In view of \eqref{ER}, we get that $(v_{\ell}, R_{\ell})$ is a smooth solution to the Euler-Reynolds system  
\begin{equation*}
\begin{cases}
\partial_{t} v_{\ell} + \div( v_{\ell} \otimes v_{\ell}) + \nabla p_{\ell} = \div R_{\ell} \\
\div v_{\ell}=0 \,.
\end{cases}
\end{equation*}

The choice of $\ell$ guarantees that both contributions in $R_{\ell}$, the mollification of $R_{q}$ and the commutator, are of equal size. In particular, $R_{\ell}$ will be of the size of $R_{q}$. More precisely we have the following

\begin{proposition}\label{p:mollifaction}
For any $N\geq 0$ it holds
\begin{align}
\norm{v_{\ell}-v_{q}}_{0} &\lesssim \delta_{q}^{\sfrac{1}{2}} \lambda_{q} \ell, \label{e:vl-vqC0}  \\
\norm{v_{\ell}}_{N+1} &\lesssim \delta_{q}^{\sfrac{1}{2}} \lambda_{q} \ell^{-N}, \label{e:vlCN} \\
\norm{R_{\ell}}_{N+\alpha} &\lesssim  \delta_{q+1}\lambda_{q}^{-\gamma-3\alpha} \ell^{-N-\alpha}.\label{e:RlCN} 
\end{align}
\end{proposition}
Here (and in what follows) the symbol $\lesssim$ means that the constant in the inequality may depend on the number of derivatives $N$, but not on any of the parameters of Proposition \ref{p:iterativeprop}, neither on the step $q$.
\begin{proof} The estimates \eqref{e:vl-vqC0} and \eqref{e:vlCN} follow from standard mollification estimates and \eqref{v_q:C1_est}. Indeed, we have
\begin{align*}
\norm{v_{\ell}-v_{q}}_{0} &\lesssim \ell \norm{v_{q}}_{1} \lesssim \delta_{q}^{\sfrac{1}{2}} \lambda_{q} \ell \,,\\
\norm{ v_{\ell}}_{N+1} &\lesssim \ell^{-N} \norm{v_{q}}_{1}\lesssim \delta_{q}^{\sfrac{1}{2}} \lambda_{q} \ell^{-N} \quad \forall N \geq 0 \,.
\end{align*}
Finally, using Proposition \ref{p:CET}, \eqref{R_q:C0_est}--\eqref{v_q:C1_est} and the choice of $\ell$ in \eqref{e:l} we get
\begin{align*}
\norm{R_{\ell}}_{N+\alpha} \lesssim  \ell^{-N-\alpha} \norm{R_{q}}_{0} + \ell^{2-N-\alpha} \norm{v_{q}}_{1}^2 
\lesssim  \ell^{-N-\alpha} \left(\delta_{q+1} \lambda_{q}^{-\gamma-3\alpha} + \ell^{2} \delta_{q} \lambda_{q}^2 \right)
\lesssim \delta_{q+1}\lambda_{q}^{-\gamma-3\alpha} \ell^{-N-\alpha} \,.
\end{align*}
\end{proof}

To localize the glued Reynolds stress in time, we use the strategy of \cite{BCV19}. By the inductive hypothesis \ref{Bq:vi}, the real bad set $\widehat{\mathcal{B}}_q$, where the Reynolds stress $R_q$ is supported, is a finite union of disjoint intervals of length $3 \tau_q$. We will split each of this intervals in subintervals $[t_i, t_{i+1}]$ of length $t_{i+1}-t_i=\theta_{q+1}$ and we will build smooth solutions $v_i$ of the Euler system with initial datum $v_i(t_i)=v_\ell(t_i)\,.$ By the choice of $\theta_{q+1}$, we have for large enough $a$ that 
\begin{equation}\label{e:existencetime}
2\theta_{q+1} \norm{v_\ell(t_i)}_{1+\alpha} \ll 1,
\end{equation}
which guarantees that $v_i$ will exist for times $\lvert t-t_i \rvert \leq 2\theta_{q+1}$ (see Proposition \ref{p:local:Euler}). This will allow us to define $\overline v_q$ as the following gluing of smooth solutions 
\begin{equation}
\overline v_q := 	 \sum_i \eta_{g}^i v_i + (1-\eta_g) v_q\,,
\end{equation}
where $\eta_g=\sum_i \eta_{g}^i$ is a smooth temporal cutoff between  $\widehat{\mathcal{B}}_q$ and $\mathcal{B}_q$. The cutoffs $\eta_{g}^i$ will be supported in an interval of length $<2\theta_{q+1}$ around $t_i$ and will be steep: $\partial_t \eta_{g}^i$ will be supported in two tiny (compared to their support) intervals $I_{i}$ and $I_{i+1}$ of length $\tau_{q+1}$ (see Section \ref{s:glueingandloc}). By construction, $\overline v_q$ will solve an Euler Reynolds system with a Reynolds stress $ \overline R_q$ which is localized where $\partial_t \eta_{g}$ is non-zero, that is in $\bigcup_i I_i\,.$ Up to enlarging every $I_i$ in length by $2\tau_{q+1}$ on either side, those intervals will form the new bad set $\mathcal{B}_{q+1}\,.$ 

\begin{proposition} \label{p:glueing} Given a couple $(v_q, R_q)$ solving \eqref{ER} on $\T^3 \times [0, T]$ together with a set $\mathcal B_q \subset [0, T]$ satisfying the hypothesis of Proposition \ref{p:iterativeprop},  there exists a smooth solution $\left(\overline v_{q}, \overline R_{q}\right)$ to \eqref{ER} on $\T^3 \times [0,T]$ and an open set $\mathcal{B}_{q+1} \subset [0,T]$ satisfying the properties {\normalfont \ref{Bq:i}--\ref{Bq:iv}} listed in Section \ref{s:inductive_prop} with $q$ replaced by $q+1$, such that additionally
\begin{align}\label{e:pglueing1}
\overline v_{q}(t)&= v_{q}(t) &&\, \text{ for all } t \in \mathcal{G}_{q}, \\
\overline R_{q}(t) &=0 &&\, \text{ for all } t \in [0,T] \text{ such that } \dist\left(t, \mathcal{G}_{q+1}\right) \leq 2 \tau_{q+1},. \label{e:pglueing2}
\end{align}
Moreover, we have the estimates
\begin{align}
\norm{\overline v_{q}- v_{q}}_{0} &\lesssim \delta_{q+1}^{\sfrac{1}{2}}  \lambda_q^{-\sfrac{\gamma}{2}- \sfrac{3\alpha}{2}}, \label{e:pglueing4} \\
\norm{\overline v_{q}- v_\ell}_{N+1} &\lesssim \delta_q^{\sfrac{1}{2}} \lambda_q \ell^{-N} \quad  &&\forall N \geq 0 \,, \label{e:pglueing5} \\
\norm{\overline v_{q}}_{N+1} &\lesssim \delta_{q}^{\sfrac{1}{2}} \lambda_{q} \ell^{-N} \quad  &&\forall N \geq 0\,,  \label{e:pglueing6}\\
\norm{\overline R_{q}}_{N+\alpha} &\lesssim \delta_{q+1}\ell^{-N+\alpha}  \quad  &&\forall N \geq 0\,, \label{e:pglueing7}\\
\norm{(\partial_{t} + \overline v_{q} \cdot \nabla) \overline R_{q}}_{N+\alpha} &\lesssim  \delta_{q+1} \delta_{q}^{\sfrac12}\lambda_q^{1+\gamma} \ell^{-N-2\alpha}  \label{e:pglueing8} \quad &&\forall N \geq 0 \,.
\end{align}
\end{proposition}
Observe that we are not explicitly requiring the properties \ref{Bq:v}--\ref{Bq:vii} on the new bad set $\mathcal{B}_{q+1}$; the latter will be however an easy consequence of the stronger properties \eqref{e:pglueing1}, \eqref{e:pglueing2} and \eqref{e:pglueing6}.

\subsection{Perturbation step}\label{s:convexintegration} Although the gluing step allows to localize the Reynolds stress $\overline R_q$ already in much smaller intervals of time, we did not improve the size of the Reynolds stress yet. In fact, the estimates have been even worsened by the factor $\lambda_q^\gamma$, which can be view as the main reason why the Hausdorff dimension achieved in Theorem \ref{t:main} is strictly bigger than the optimal one given in Conjecture \ref{conject}.  The precise discussion of this issue is postponed to Section \ref{sec_gap} below.

In order reduce the size of the Reynolds stress, we will produce from $(\overline v_q, \overline R_q)$ and $\mathcal{B}_{q+1}$, given by Proposition \ref{p:glueing}, a new solution $(v_{q+1}, R_{q+1})$ to \eqref{ER} with the Reynolds stress $R_{q+1}$ still supported in $\widehat{\mathcal{B}}_{q+1}$ and verifying all the desired estimates. This will be done by adding a highly oscillatory perturbation $w_{q+1}$ to $\overline v_q$. Indeed, this is the key ingredient of all convex integration schemes building on \cite{DS2013} and,  as in \cite{Is2018, BDLSV2019}, the building blocks for the perturbation $w_{q+1}$ are suitably chosen, stationary solutions to the Euler equations, the so called Mikado flows. In the presentation of the perturbation step, we will follow closely \cite{BDLSV2019}.

At difference from \cite{BDLSV2019}, we will need to localize the perturbation $w_{q+1}$ in time to have support within $\widehat{\mathcal{B}}_{q+1}$. This will be achieved by means of steep temporal cutoffs, similar to the ones from the gluing step,  and will be responsible for worsened estimates on  $R_{q+1}$ with respect to \cite{BDLSV2019}.

\begin{proposition}\label{p:perturbation}
Let $(\overline v_q,\overline R_q)$ and the bad set $\mathcal{B}_{q+1}\subset [0,T]$ be as in Proposition \ref{p:glueing}. There exists a new smooth couple $(v_{q+1}, R_{q+1})$ which solves \eqref{ER} in $\T^3\times[0,T]$, and such that all the properties {\normalfont \ref{Bq:i}--\ref{Bq:vii}} listed in Section \ref{s:inductive_prop} hold with $q$ replaced by $q+1$. Moreover, we have the estimates
\begin{align}
\norm{ v_{q+1}- \overline v_{q}}_{0} + \lambda_{q+1}^{-1} \norm{  v_{q+1}- \overline v_{q}}_{1} &\leq \frac{M}{2} \delta_{q+1}^{\sfrac{1}{2}}, \label{v_q+1-bar v_q} \\
\norm{ R_{q+1}}_{0} &\lesssim\frac{ \delta^{\sfrac{1}{2}}_{q+1}\delta_{q}^{\sfrac12}\lambda_q^{1+\gamma}}{\lambda_{q+1}^{1-5\alpha}}\label{R_q+1},
\end{align}
where $M>0$ is a universal geometric constant.
\end{proposition}
The proof of the previous proposition is the core of the convex integration scheme and will occupy most of this work. Being quite technical, we postpone it to Section \ref{s:perturbation}.

\subsection{Proof of Proposition \ref{p:iterativeprop}} We prove how the main iterative Proposition \ref{p:iterativeprop} is a consequence of the two previous steps and we postpone their respective proofs in Sections \ref{s:glueingandloc} and \ref{s:perturbation} below. 

We start by noticing that given a couple $(v_q, R_q)$ solving \eqref{ER} on $\T^3 \times [0, T]$ together with a set $\mathcal B_q \subset [0, T]$ satisfying the hypothesis of Proposition \ref{p:iterativeprop}, Proposition \ref{p:perturbation} directly gives the smooth couple $(v_{q+1},R_{q+1})$ solving \eqref{ER} on $\T^3 \times [0, T]$, together with the bad set $\mathcal{B}_{q+1}$ (and consequently its complement $\mathcal{G}_{q+1}$), satisfying properties \ref{Bq:i}--\ref{Bq:vii}. Thus we are left to check \eqref{v_q+1-v_q} and that estimates \eqref{R_q:C0_est}--\eqref{vq:C0_est} hold with $q$ replaced by $q+1$.

Estimate \eqref{v_q+1-v_q} is a consequence of \eqref{e:pglueing4}, \eqref{e:pglueing6}, \eqref{v_q+1-bar v_q} and the inductive assumption \eqref{v_q:C1_est} on $v_q$. Indeed, we have
\begin{equation}\label{ineq_1}
\|v_{q+1}-v_q\|_0\leq \|v_{q+1}-\overline v_q\|_0+\| \overline v_q-v_q\|_0\leq\delta_{q+1}^{\sfrac12} \left(\frac{M}{2}+C \lambda_q^{-\sfrac{\gamma}{2}}\right)\leq M\delta_{q+1}^{\sfrac12},
\end{equation}
where the last inequality holds if $a$ is  large enough. Similarly, we get $a$ large enough (independently of $q$) that
\begin{equation}\label{ineq_2}
\|v_{q+1}-v_q\|_1\leq \|v_{q+1}-\overline v_q\|_1+\| \overline v_q-v_q\|_1\leq\delta_{q+1}^{\sfrac12}\lambda_{q+1} \left(\frac{M}{2}+C \frac{\delta_q^{\sfrac{1}{2}}\lambda_q}{\delta_{q+1}^{\sfrac{1}{2}}\lambda_{q+1}}\right)\leq \frac23 M\delta_{q+1}^{\sfrac12}\lambda_{q+1},
\end{equation}
which together with \eqref{ineq_1}, proves \eqref{v_q+1-v_q}.

By using \eqref{v_q:C1_est} and \eqref{ineq_2}, we obtain 
$$
\|v_{q+1}\|_{1}\leq \|v_{q+1}-v_q\|_1+\|v_q\|_1\leq  \frac23 M\delta_{q+1}^{\sfrac12}\lambda_{q+1}+M\delta_q^{\sfrac12}\lambda_q\leq M\delta_{q+1}^{\sfrac12}\lambda_{q+1},
$$
which ensures the validity of \eqref{v_q:C1_est} at step $q+1$. Moreover, 
$$
\|v_{q+1}\|_0\leq \|v_{q+1}-v_q\|_0+\|v_q\|_0 \leq M \delta_{q+1}^{\sfrac12}+1-\delta_q^{\sfrac12} \leq 1-\delta_{q+1}^{\sfrac12},
$$
where again we assumed that $a$ is large enough in order to guarantee the last inequality. Thus also \eqref{vq:C0_est} holds at step $q+1$. Finally, the proof of the last estimate \eqref{R_q:C0_est} is a consequence of the following relation
\begin{equation}\label{param_ineq}
\frac{\delta_{q+1}^{\sfrac12}\delta_q^{\sfrac12}\lambda_q^{1+\gamma}}{\lambda_{q+1}^{1-5\alpha}}\leq \delta_{q+2}\lambda_{q+1}^{-\gamma-3\alpha}.
\end{equation}
By using the parameters definitions, it is clear that \eqref{param_ineq} holds if
\begin{equation}\label{param_ineq_2}
-\beta b - \beta +1+\gamma-b+5\alpha b<-2\beta b^2-\gamma b -3\alpha b.
\end{equation}
We notice that if the previous inequality holds for $\alpha=0$, then (being an inequality between polynomials) there will be an $\alpha_0=\alpha_0(\gamma,\beta,b)>0$ such that \eqref{param_ineq_2} still holds for all $0<\alpha<\alpha_0$. But if we set $\alpha=0$, we obtain 
$$
(b+1)\gamma<-2\beta b^2+(\beta+1)b+\beta-1=(b-1)(-2\beta b +1-\beta),
$$
which holds by our choice of $\gamma$ in \eqref{gamma:bound2}. This concludes the proof of Proposition \ref{p:iterativeprop}.

\subsection{Gap with the conjectured exponent}\label{sec_gap}By our inductive assumption \ref{Bq:iv} on the shrinking rate of the bad set $\mathcal{B}_q$, it is clear that bigger is $\gamma$, the smaller is the dimension of the final singular set. By looking at \eqref{dim_est_B}, one may verify that the sharp Hausdorff dimension of Conjecture \ref{conject} would be achieved if $\gamma$ could reach the threshold $\overline \gamma (b,\beta)\simeq (b-1)(1-\beta-2\beta b)$. More precisely, we would need that this upper bound $\overline \gamma$ satisfies 
\begin{equation}\label{gamma_sharp}
\lim_{b\rightarrow 1}\frac{\overline{\gamma}(b,\beta)}{b-1}=1-3\beta.
\end{equation}
In our case however, the restriction \eqref{gamma:bound2} on $\gamma$ implies that the maximal $\gamma_{\text{max}}(b,\beta)$ we can choose is only half of the sharp one from \eqref{gamma_sharp}, or in other words, our upper bound on $\gamma$ satisfies
$$
\lim_{b\rightarrow 1}\frac{\gamma_{\text{max}}(b,\beta)}{b-1}=\frac{1-3\beta}{2}.
$$
With that being said, we will now try to explain where the restriction \eqref{gamma:bound2} comes from. To do that, we give an heuristic version of the inductive scheme.

Given the two parameters $\delta_q$ and $\lambda_q$ as in Section \ref{s:inductive_prop}, the aim is to find a perturbation $w_{q+1}$ of size $\delta_{q+1}$, oscillating at frequency $\lambda_{q+1}\,,$ that verifies \eqref{v_q+1-v_q}, together with a new error $R_{q+1}$ that is localized in intervals of length $\sim \tau_{q+1}$. By looking at the oscillation error in \eqref{d:R_q+1}, we deduce that $\|w_{q+1}\|_0\simeq \norm{\overline R_q}^{\sfrac12}_0$, since without that, it would be impossible to ensure that $R_{q+1}$ is considerably smaller than $R_q$. This implies that 
\begin{equation}\label{size_glued_R}
\norm{\overline R_q}_0\simeq\delta_{q+1}.
\end{equation}
The stress tensor $\overline R_q$ is obtained by using the gluing technique introduced by Philip Isett in \cite{Is2018} and consequently, the corresponding glued velocity $\overline v_q$ has to be an exact solution of Euler in intervals of length $\theta_{q+1}\simeq (\delta_q^{\sfrac{1}{2}} \lambda_q)^{-1}\,.$  The only difference is that we need to shrink the temporal support of $\overline R_q$ to intervals of length $\sim\tau_{q+1}=\lambda_q^{-\gamma} \theta_{q+1} \ll \theta_{q+1}$ (see Section \ref{s:glueingandloc} for the detailed construction). This asymmetry between the two sizes implies $\norm{\overline R_q}_0\simeq \|R_q\|_0 \lambda_q^\gamma$  (see estimates \eqref{R_q:C0_est} and \eqref{e:pglueing7}), which together with \eqref{size_glued_R}, forces
\begin{equation}
\label{size_R_q}
\|R_q\|_0\simeq \delta_{q+1}\lambda_{q}^{-\gamma}
\end{equation}
to be the right inductive assumption at step $q$. The perturbation $w_{q+1}$ can now cancel the error $\overline R_q$, but we still need to force $\left|\supp_t R_{q+1} \right|\simeq \tau_{q+1}$. By looking at the definition of the new Reynolds stress in \eqref{d:R_q+1}, the easiest way is to localize the perturbation $w_{q+1}$ in such intervals by means of steep temporal cut-offs; that is by setting
$$
w_{q+1}:=\eta_p \tilde w_{q+1},
$$
where $\tilde w_{q+1}$ is a combination of highly oscillatory (at frequency $\lambda_{q+1}$) Mikado flows and $\eta_p$ is a time cut-off such that $\left|\supp_t \eta_p \right|\simeq \tau_{q+1}$. This of course implies that $\|\partial_t \eta_p\|_0\lesssim \tau_{q+1}^{-1}$  (see Lemma \ref{l:cutoff_pert}). In this way, we are inserting in $R_{q+1}$ (in particular in the transport error $R_{transp}$) a term that looks like $\partial_t \eta_p \mathcal R \tilde w_{q+1}$, which from Proposition \ref{p:R:oscillatory_phase} satisfies 
\begin{equation}
\label{temp_error}
\left\|\partial_t \eta_p \mathcal R \tilde w_{q+1} \right\|_0\lesssim \frac{\delta_{q+1}^{\sfrac12}}{\lambda_{q+1}}\tau_{q+1}^{-1}\simeq \frac{\delta_{q+1}^{\sfrac12}\delta_q^{\sfrac12} \lambda_q^{1+\gamma}}{\lambda_{q+1}}.
\end{equation}
Finally, to close the inductive estimate, we need to check that the bound in \eqref{temp_error} is below the one inductively assumed in \eqref{size_R_q}, at step $q+1$, which is 
\begin{equation}
\label{bad_relation}
 \frac{\delta_{q+1}^{\sfrac12}\delta_q^{\sfrac12} \lambda_q^{1+\gamma}}{\lambda_{q+1}}\leq \delta_{q+2}\lambda_{q+1}^{-\gamma}.
\end{equation}
As already done in \eqref{param_ineq}, the previous relation gives the bound \eqref{gamma:bound2} on $\gamma$.

To end the discussion, we believe that a possible way to prove Conjecture \ref{conject}, could be to find a way to construct $\overline R_q$ such that there is no $\lambda_q^{\gamma}$ loss in size with respect to $R_q$, or more explicitly
$$
\norm{\overline R_q}_0\simeq \|R_q\|_0 \quad \text{and} \quad \left|\supp_t \overline R_q \right|\simeq \tau_{q+1}.
$$
In this way, the inductive hypothesis on the Reynolds stress becomes $\|R_q\|_0\lesssim \delta_{q+1}$, and consequently, the $\lambda_{q+1}^{-\gamma}$ disappears from the right hand side of \eqref{bad_relation}, allowing $\gamma$  to reach the sharp threshold $\overline \gamma(b,\beta)\simeq (b-1)(1-\beta-2\beta b)$.

\section{Gluing and localization step: Proof of Proposition \ref{p:glueing}}\label{s:glueingandloc}

In this section, we prove Proposition \ref{p:glueing}. To make our construction compatible with the choice of all parameters, we choose $a$  large enough (depending on all the parameters $\beta, b, \gamma, \alpha$, but not on $q$) such that\footnote{ Indeed, in order to guarantee the first inequality, it suffices to require $5\lambda_q^{-\gamma} < 1$ which is enforced if $a$ is such that $10\pi a^{-\gamma} < 1\,.$ To ensure the second inequality, we observe that by \eqref{gamma:bound1}, we have for $\alpha$ small enough 
$$2\frac{\theta_{q+1}}{\tau_{q}}\lesssim a^{b^{q-1}(\gamma-(b-1)(1+3\alpha-\beta))} \lesssim a^{-((b-1)(1+3\alpha-\beta)-\gamma)/b}\ll 1\,.$$}
\begin{equation}\label{e:timecompatibility}
5 \tau_{q+1} < \theta_{q+1} \quad \text{ and } \quad 2 \theta_{q+1} < \tau_q \,.
\end{equation} 

We now fix a couple $(v_q, R_q)$ solving \eqref{ER} on $\T^3 \times [0, T]$ together with a set $\mathcal B_q \subset [0, T]$ satisfying the hypothesis of Proposition \ref{p:iterativeprop}. By assumption \ref{Bq:vi} on $\mathcal{B}_q$, we can write $\mathcal{B}_q$ as a disjoint union of finitely many open intervals $J_q$ of length $5 \tau_q$ and consequently, by setting $\widehat{J}_q := \left\{ t \in J_q: \dist(t, \mathcal{G}_q)>\tau_{q} \right\}$, we can write $\widehat{\mathcal{B}}_q$  as disjoint union of intervals $\widehat J_q$ of length $3\tau_{q}\,.$ Observe that by \eqref{e:timecompatibility}, we have $\dist(\widehat{J}_q, \mathcal{G}_q) = \tau_{q} > 2 \theta_{q+1}\,.$

For every such interval $J_q$, we will first construct a smooth solution $\left(\overline v_q, \overline R_q\right)$ to \eqref{ER} on $\T^3 \times J_q$ and equidistributed intervals $\{ I_i \}_{i=0}^{n+1}\subset J_q $ with 
$ n:= \ceil*{ \frac{3\tau_q}{\theta_{q+1}}}$ such that
\begin{enumerate}[label=\normalfont(\alph*)]
\item\label{intervals} $\dist(I_i, I_{i+1})= \theta_{q+1}\,, $ $\lvert I_i \rvert  = \tau_{q+1}\,,$ and all the $I_i$ lie in a $2\theta_{q+1}$ neighbourhood of $\widehat{J}_q$, that is
$$\bigcup_{i=0}^{n+1} I_i \subset \left\{ t \in J_q: \dist\left(t, \widehat{J}_q\right) < 2\theta_{q+1} \right\}\,,$$
\item\label{support} $\supp \overline R_q \subset \T^3 \times \bigcup_{i=0}^{n+1} I_i\,,$
\item\label{compatability} $\overline v_q(t)= v_q(t) \, \quad \forall t \in \left \{ t \in J_q: \dist\left(t, \widehat{J}_q\right) \geq 2 \theta_{q+1} \right\} \,,$
\item \label{estimates}$\left(\overline v_q, \overline R_q\right)$ satisfies the estimates \eqref{e:pglueing4}--\eqref{e:pglueing8} when restricted to times  $t \in J_q\,.$
\end{enumerate} 

Properties \ref{support} and \ref{compatability} allows to extend the different $\left(\overline v_q, \overline R_q\right)$ (coming from different intervals $J_q$) to a smooth couple $\left(\overline v_q, \overline R_q\right)$ solving \eqref{ER} on $\T^3 \times [0,T]$ by setting 
\begin{equation*}
\overline v_q(t):= v_q(t) \text{ and } \overline  R_q(t) := 0 \quad   \forall t \in \mathcal{G}_q= [0,T] \setminus  \mathcal{B}_q\,.
\end{equation*}
By construction, $\overline v_q$ satisfies \eqref{e:pglueing1}. The new bad set $\mathcal{B}_{q+1}$ is obtained by enlarging the intervals $I_i$ by $2 \tau_{q+1}$ on either side;
\begin{equation}
\mathcal{B}_{q+1}:= \left\{ t \in [0,T] \,:  t \in J_q \text{ for one of the disjoint intervals of } \mathcal{B}_q \text{ and }\dist\left(t,  \bigcup_{ i=0}^{n+1} I_i \right) < 2 \tau_{q+1} \right\} \,.\label{def_Bq+1}
\end{equation}
Using \eqref{e:timecompatibility} and \ref{intervals}, it is easy to see that $\mathcal{B}_{q+1}$ is made of disjoint intervals of length $5 \tau_{q+1}$ and $\mathcal{B}_{q+1}\subset  \bigcup J_q= \mathcal{B}_q$. In particular, the new bad set satisfies the properties \ref{Bq:i}--\ref{Bq:iii}, and it remains to verify \ref{Bq:iv}. By construction 
\begin{equation*}
\lvert \mathcal{B}_{q+1} \rvert = 5 \tau_{q+1} \frac{\lvert \mathcal{B}_{q}\rvert}{5\tau_q}  \left( \ceil*{ \frac{3\tau_q}{\theta_{q+1}}} +2 \right) \leq 10 \lvert \mathcal{B}_q \rvert \frac{\tau_{q+1}}{\theta_{q+1}} \,,
\end{equation*}
where in the last inequality, we are assuming $\alpha$ small and $a$ large enough. Thus also \ref{Bq:iv} holds true. Finally, with this definition of $\mathcal{B}_{q+1}$, the property \eqref{e:pglueing2} is an immediate consequence of \ref{support}, and the estimates \eqref{e:pglueing4}--\eqref{e:pglueing8} are a consequence of \ref{estimates}. Indeed, for times $t \in \mathcal{B}_q$ the estimates hold by \ref{estimates}. For $t\in \mathcal{G}_q=[0,T] \setminus \mathcal{B}_q$ we have $\overline v_q(t)=v_q(t)$ and $\overline R_q(t) \equiv 0$ which makes estimates \eqref{e:pglueing4}, \eqref{e:pglueing7} and \eqref{e:pglueing8} trivial. Estimates \eqref{e:pglueing5} and \eqref{e:pglueing6} hold then by triangular inequality, \eqref{e:vlCN} and assumption \ref{Bq:vii} (see also the remarks preceeding \eqref{e:glueingbetterest}). 

For the rest of this section, we thus fix one of the intervals $J_q$ and the corresponding $\widehat{J}_q\,.$

\subsection{Construction of $(\overline v_q, \overline R_q)$} We start by picking the equidistributed times $ t_0 <t_1< \dots < t_n$ by setting $t_0$ to be the left endpoint of the interval $\widehat{J}_q$ and by setting inductively $t_{i+1}:=  t_i+ \theta_{q+1}$ until reaching 
\begin{equation}
n:= \ceil*{ \frac{\lvert \widehat{J}_q \rvert}{\theta_{q+1}}} =\ceil*{\frac{3 \tau_q}{\theta_{q+1}}} \,. 
\end{equation}
In other words, $t_{n}$ is the right endpoint of $\widehat{J}_q$ in case $\theta_{q+1}$ happens to be a multiple of $\tau_q$, otherwise it is the first time falling thereafter. This procedure is compatible with the choice of parameters by \eqref{e:timecompatibility}. For each $i=0, \dots, n\,,$ we now consider the smooth solutions $(v_i, p_i)$ of the Euler equations with initial datum $v_\ell(t_i)$ defined on their maximal time of existence, that is
\begin{equation}
\begin{cases}
\partial_t v_i + \div(v_i \otimes v_i) + \nabla p_i = 0 \\
\div v_i=0 \\
v_i(\cdot, t_i)= v_\ell (t_i) \,,
\end{cases} 
\end{equation}
where $v_\ell$ is the spatial mollification of $v_q$ at length scale $\ell=\ell_q$ defined in \eqref{e:l}. From Proposition \ref{p:local:Euler}, \eqref{e:existencetime} and \eqref{e:vlCN}, it follows that each $v_i$ exists for times $\lvert t-t_i \rvert \leq 2\theta_{q+1}$ and enjoys the estimate
\begin{equation}\label{e:viCN}
\norm{ v_i (t)}_{N+\alpha} \lesssim \norm{v_\ell(t_i)}_{N+\alpha} \lesssim \delta_q^{\sfrac{1}{2}}\lambda_q \ell^{1-N-\alpha}, \quad \text{ for } \lvert t-t_i \rvert \leq 2 \theta_{q+1} \text{ and } N \geq 1  \,.
\end{equation}

We now glue the exact solutions $v_i$ by means of steep cutoffs $\eta_g^i$ centered in $t_i$ which are constructed in the following

\begin{lemma}\label{l:cutoffglueing} Let $\widehat{J}_q$ be one of the disjoint intervals $\widehat{B}_q$ is made of and let $t_0 <t_1 < \dots<t_n$ be the equidistributed points picked before. There exists a family of cutoff $\{ \eta_g^{i} \}_{i=0}^{n} \in C^\infty_c((0,T))$ such that
\begin{enumerate}[label=\normalfont(\alph*)]
\item\label{eta:a} $\eta_g(t):= \sum_{i=0}^{n} \eta_g^i (t) =1 \quad \forall t \in \widehat{J}_q \,,$ 
\item\label{eta:b} $\eta_g^i$ is supported in an interval of size $< 2\theta_{q+1}$ centered at $t_i\,.$ More precisely,
\begin{equation*}
\supp \eta_g^i \subset \left(t_i - \frac{\theta_{q+1}+ \tau_{q+1}}{2}, t_i + \frac{\theta_{q+1}+ \tau_{q+1}}{2}\right)\,,
\end{equation*}
\item\label{eta:c} $0 \leq \eta_g^i \leq 1$ and 
\begin{equation*}
\eta_g^i(t)=1 \quad \forall t \in \left[t_i - \frac{\theta_{q+1}- \tau_{q+1}}{2}, t_i + \frac{\theta_{q+1}- \tau_{q+1}}{2}\right]  \,,
\end{equation*}
\item\label{eta:e}  $\norm{\partial_t^N \eta_{g}^i}_{0} \lesssim \tau_{q+1}^{-N} \quad \forall N \geq 0 \,.$
\end{enumerate}
\end{lemma}

Observe that by construction, $\partial_t \eta_{g}$ is supported strictly inside $\bigcup_{i={0}}^{n+1} I_i$, with
\begin{equation*}\label{def:Ii}
I_i:= \left( t_i-\frac{\theta_{q+1} +\tau_{q+1}}{2},  t_i-\frac{\theta_{q+1} +\tau_{q+1}}{2}\right),
\end{equation*}
and that from \ref{eta:a} and \ref{eta:b}, we have
\begin{equation*}
\eta_g^{i-1}(t)=1-\eta_g^i (t)  \quad \forall t \in I_i \quad \text{ and } \quad  i=1, \dots, n\,.
\end{equation*}

Since $\supp \eta_g^i \subset \{ t: \lvert t-t_i \rvert \leq 2 \theta_{q+1} \}\,,$ the following gluing of exact solutions is well-defined
\begin{align}\label{def:barvq}
\overline v_q(x,t) &:= \sum_{i=0}^n \eta_g^i(t) v_i(x,t) + (1-\eta_g(t)) v_q(x,t) \quad \text{ for } t \in J_q \\
\overline p_q(x,t)&:=\sum_{i=0}^n \eta_g^i(t) p_i(x,t) + (1-\eta_g(t)) p_q(x,t) \quad \text{ for } t \in J_q  \label{def:barpq}\,.
\end{align}
 
It follows that $\overline v_q$ is smooth and is an exact solution to Euler outside $\bigcup_{i=0}^{n+1} I_i$; more precisely
\begin{align*}
\partial_t \overline v_q + \div( \overline v_q &\otimes \overline v_q) + \nabla \overline p_q \\
&= \begin{cases}0 \quad \text{ in } J_q \setminus \bigcup_{i=0}^{n+1} I_i\,,&\\
\partial_t \eta_g^i\, (v_i-v_{i-1}) - \eta_g^i(1-\eta_g^i) \div ( (v_i-v_{i-1}) \otimes (v_i-v_{i-1})) &\quad \text{ in } I_i \text{ for } i \in \{ 1, \dots, n \},\\
\partial_t \eta_g^0\, (v_0-v_q) - \eta_g^0(1-\eta_g^0) \div ( (v_0-v_q) \otimes (v_0 - v_{q}))  &\quad \text{ in } I_0, \\
\partial_t \eta_g^n\,  (v_n-v_q) - \eta_g^n(1-\eta_g^n) \div ( (v_n-v_q) \otimes (v_n - v_{q}))  &\quad \text{ in } I_{n+1}.
\end{cases}
\end{align*}
Recall from \eqref{e:R:def} the inverse divergence operator $\mathcal{R}$ acting on vector fields with zero average. Since all the $v_i$ and $v_q$ have all the same average, we can define the new localized Reynolds stress by
\begin{align}\label{def:barRq}
\overline R_q := 
\begin{cases}
0 \quad \text{ in } \text{ in } J_q \setminus \bigcup_{i=0}^{n+1} I_i \,,\\
\partial_t \eta_g^i \, \mathcal{R} (v_i-v_{i-1}) - \eta_g^i(1-\eta_g^i) (( v_i-v_{i-1}) \otimes (v_i-v_{i-1})) \quad &\text{ in } I_i \text{ for } i \in \{ 1, \dots,  n \}\,,\\
\partial_t \eta_g^0 \, \mathcal{R} (v_0-v_q) - \eta_g^0(1-\eta_g^0) ((v_0-v_q) \otimes (v_0 - v_{q})) \quad &\text{ in } I_0, \\
\partial_t \eta_g^n \, \mathcal{R} (v_n-v_q) - \eta_g^n(1-\eta_g^n) ((v_n-v_q) \otimes (v_n - v_{q})) \quad &\text{ in } I_{n+1}.
\end{cases}
\end{align}
With this definition, the smooth couple $\left(\overline v_q, \overline R_q\right)$ solves \eqref{ER} on $\T^3 \times J_q$ and has already the desired localization property
\begin{equation}
\supp \overline R_q \subset \T^3 \times \bigcup_{i=0}^{n+1} I_i \quad \text{ with } \lvert I_i \rvert= \tau_{q+1}  \text{ and } n:= \ceil*{\frac{3\tau_q} {\theta_{q+1}}}\,.
\end{equation}

\begin{figure}
\begin{center}
 \includegraphics[height=0.4\textheight]{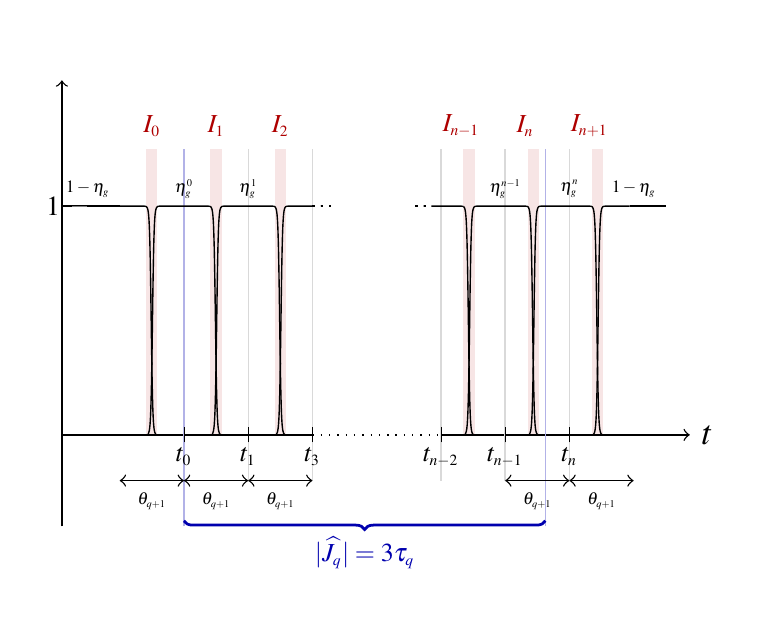}
\caption{ Gluing and localization procedure in one of the intervals of $\widehat{\mathcal{B}}_q$.}
\end{center}
\end{figure}

\subsection{Stability estimates on $v_i-v_\ell$ and improved bounds on $v_\ell-v_q$ on $\widehat{\mathcal{B}}_q^c$} 
We first establish stability estimates on two adjacent exact smooth solutions of Euler, $v_i$ and $v_{i+1}$. Since  $(v_i-v_{i+1}) = (v_i- v_\ell) + (v_\ell- v_{i+1})\,,$ it suffices to estimate $v_i-v_\ell\,.$ The proof of the following proposition follows closely \cite{BDLSV2019} with some minor changes. Here (and in what follows), we denote the material derivative by 
\begin{equation}\label{e:materialderivative}
D_{t, \ell}:= \partial_t + (v_\ell \cdot \nabla)
\end{equation}

\begin{proposition}\label{p:stabilityvi} For $\lvert t- t_i \rvert \leq 2\theta_{q+1}$  we have the estimates 
\begin{align}
\norm{(v_i- v_\ell)(t)}_{N+\alpha} &\lesssim \tau_{q+1} \delta_{q+1} \ell^{-N-1+\alpha} \quad &&\forall N \geq 0 \,, \label{e:vi-vlCN}\\
\norm{\nabla(p_i-p_\ell)(t)}_{N+ \alpha} &\lesssim \lambda_{q}^{-\gamma} \delta_{q+1} \ell^{-N-1+\alpha} \quad &&\forall N \geq 0 \,,\label{e:pi-plCN}\\
\norm{ D_{t,\ell} (v_i-v_\ell)(t)}_{N+\alpha}&\lesssim \lambda_{q}^{-\gamma} \delta_{q+1} \ell^{-N-1+\alpha} \quad &&\forall N \geq 0\,. \label{e:vi-vlmatCN}
\end{align}
\end{proposition}
%REMARK: We could have +2 \alpha in the last two estimates.

\begin{proof}[Proof of Proposition \ref{p:stabilityvi}] Observe that $(v_\ell-v_i)$ is divergence-free and it solves
\begin{equation}\label{e:vl-vi}
\partial_t (v_\ell-v_i) + (v_\ell \cdot \nabla) (v_\ell-v_i) + \nabla (p_\ell-p_i)= ((v_i-v_\ell)\cdot \nabla)v_i  + \div R_\ell \,,
\end{equation}
so that, by taking the divergence, we find the following equation for the pressure term
\begin{equation}\label{e:pl-pi}
\Delta (p_\ell-p_i) =- \div(\nabla v_\ell (v_\ell-v_i)) + \div(\nabla v_i (v_i-v_\ell))   + \div \div R_\ell \,.
\end{equation}
By Schauder estimates we get
\begin{equation}\label{e:vl-vipres}
\begin{split}
\norm{\nabla(p_\ell-p_i)}_\alpha &\lesssim \left(\norm{ \nabla v_\ell}_\alpha + \norm{\nabla v_i}_\alpha\right) \norm{v_\ell-v_i}_\alpha + \norm{\div R_\ell}_\alpha \\
&\lesssim \delta_q^{\sfrac{1}{2}} \lambda_q \ell^{-\alpha} \norm{v_\ell-v_i}_\alpha + \lambda_q^{-\gamma-3\alpha} \delta_{q+1} \ell^{-1-\alpha} \,,
\end{split}
\end{equation}
where we used \eqref{e:vlCN} and \eqref{e:viCN} (together with \ref{e:Holderinterpolation2}) as well as  \eqref{e:RlCN} in the last inequality. Hence
\begin{equation}\label{e:vl-vimat}
\norm{D_{t, \ell} (v_\ell-v_i) (t)}_\alpha \lesssim  \delta_q^{\sfrac{1}{2}} \lambda_q \ell^{-\alpha} \norm{(v_\ell-v_i)(t)}_\alpha + \lambda_q^{-\gamma-3\alpha} \delta_{q+1} \ell^{-1-\alpha}, \quad \text{if } \lvert t-t_i \rvert \leq 2 \theta_{q+1} \,.
\end{equation}
Since $(v_\ell - v_i)(t_i)=0$, we deduce from Proposition \ref{p:transport} that
\begin{equation*}
\norm{ (v_\ell-v_i)(t)}_\alpha \lesssim \int_{t_i}^t \left(\delta_q^{\sfrac{1}{2}} \lambda_q \ell^{-\alpha} \norm{(v_\ell-v_i)(\tau)}_\alpha + \lambda_q^{-\gamma-3\alpha} \delta_{q+1} \ell^{-1-\alpha}\right) \, d \tau \,.
\end{equation*}
Using that $\theta_{q+1}\delta_q^{\sfrac{1}{2}} \lambda_q \ell^{-\alpha}= \lambda_q^{-3\alpha} \ell^{-\alpha}\leq \ell^{\alpha}\leq 1 $ by \eqref{l:bound}, we deduce from Gr\"onwall's inequality that 
\begin{equation}\label{e:vl-viN=1}
\norm{ (v_\ell-v_i)(t)}_\alpha \lesssim \lambda_q^{-\gamma-3\alpha}\delta_{q+1} \ell^{-1-\alpha} \theta_{q+1} \lesssim \tau_{q+1} \delta_{q+1}  \ell^{-1+\alpha}, \quad  \text{if }  \lvert t-t_i \rvert \leq 2 \theta_{q+1} \,.
\end{equation}
Inserting this estimate in \eqref{e:vl-vipres} and \eqref{e:vl-vimat}, we have obtained the claimed estimates for $N=0\,.$ 

For $N\geq 1$, we fix a spatial derivative $\partial^\theta$ of order $\lvert \theta \rvert =N\,.$ We differentiate \eqref{e:vl-vi} and estimate, using the interpolation inequality for the H\"older norm of products \eqref{e:Holderproduct} on the nonlinear term, 
\begin{align}
\norm{ \partial^\theta D_{t, \ell} (v_\ell-v_i)}_\alpha &\lesssim \norm{\nabla(p_\ell-p_i)}_{N+\alpha} +\norm{v_\ell-v_i}_{N+\alpha} \norm{\nabla v_i}_\alpha + \norm{ v_\ell-v_i}_\alpha \norm{\nabla v_i}_{N+ \alpha}+ \norm{\div R_\ell}_{N+\alpha} \nonumber \\
&\lesssim  \norm{\nabla(p_\ell-p_i)}_{N+\alpha} +\delta_q^{\sfrac{1}{2}} \lambda_q \ell^{-\alpha} \norm{v_i-v_\ell}_{N+\alpha}  + \lambda_{q}^{-\gamma-3\alpha} \delta_{q+1} \ell^{-N-1-\alpha} \,,\label{N_der_Dtl_v_l-v_i}
\end{align}
where the second inequality is a consequence of \eqref{e:viCN}, \eqref{e:vl-viN=1} and \eqref{e:RlCN}. Reusing the equation \eqref{e:pl-pi} for the pressure term and Proposition \ref{p:CZO_C_alpha}, we have, arguing as before, that
\begin{align}
&\norm{\nabla (p_\ell-p_i)}_{N+\alpha} \nonumber \\
 &\lesssim (\norm{\nabla v_\ell}_{N+\alpha} + \norm{\nabla v_i}_{N+\alpha}) \norm{v_\ell-v_i}_\alpha + (\norm{\nabla v_\ell}_\alpha + \norm{\nabla v_i}_\alpha) \norm{v_\ell-v_i}_{N+\alpha} + \norm{\div R_\ell}_{N+\alpha} \nonumber \\
&\lesssim \delta_q^{\sfrac{1}{2}} \lambda_q \ell^{-\alpha} \norm{v_i-v_\ell}_{N+\alpha} + \lambda_q^{-\gamma-3\alpha} \delta_{q+1} \ell^{-N-1-\alpha}\,. \label{e:pi-plCNaux}
\end{align}
We now write 
$ \partial^\theta D_{t, \ell} (v_\ell- v_i)= D_{t, \ell} \partial^\theta (v_\ell-v_i) + [\partial^\theta, (v_\ell \cdot \nabla) ] (v_\ell-v_i)$ and observe, using the Leibniz rule, that the commutator $ [\partial^\theta, (v_\ell \cdot \nabla) ] (v_\ell-v_i)$ involves only spatial derivatives of order at most $N$ of $v_\ell-v_i\,.$ Using again \eqref{e:Holderproduct} to estimate all the nonlinear terms appearing in the commutator, \eqref{e:Holderinterpolation2} and Young, we have
\begin{align*}
&\norm{[\partial^\theta, (v_\ell \cdot \nabla) ] (v_\ell-v_i)}_\alpha \\
&\lesssim \sum_{k=1}^N \norm{v_\ell}_{k+\alpha} \norm{ v_\ell-v_i}_{N+1-k+\alpha}
%&\lesssim \sum_{k=1}^N \norm{v_\ell}_{1+\alpha}^\frac{N-k+1}{N} \norm{v_\ell}_{N+1+\alpha}^\frac{k-1}{N} \norm{ v_\ell-v_i}_{1+\alpha}^\frac{k-1}{N}\norm{ v_\ell-v_i}_{N+1+\alpha}^\frac{N-k+1}{n}\\
\lesssim \norm{ v_\ell}_{1+\alpha} \norm{ v_\ell-v_i}_{N+\alpha} + \norm{v_\ell}_{N+1+\alpha} \norm{v_\ell-v_i}_{\alpha} \\
&\lesssim \delta_q^{\sfrac{1}{2}} \lambda_q \ell^{-\alpha} \norm{ v_\ell-v_i}_{N+\alpha}+  \lambda_q^{-\gamma-3\alpha} \delta_{q+1} \ell^{-N-1-\alpha}\,,
\end{align*}
where the last inequality uses again \eqref{e:vlCN} and \eqref{e:vl-viN=1}. Collecting terms, we obtain
\begin{equation}\label{e:vi-vlmatCNaux}
\norm{D_{t, \ell} \partial^\theta (v_\ell-v_i)(t)}_{\alpha} \lesssim  \delta_q^{\sfrac{1}{2}} \lambda_q \ell^{-\alpha} \norm{(v_\ell-v_i)(t) }_{N+\alpha}+  \lambda_q^{-\gamma-3\alpha} \delta_{q+1} \ell^{-N-1-\alpha}, \quad  \text{if }  \lvert t-t_i \rvert \leq 2 \theta_{q+1} \,.
\end{equation}
Reusing Proposition \ref{p:transport} together with the fact that $\partial^\theta (v_\ell-v_i)(t_i)=0$, we have
\begin{equation*}
\norm{(v_\ell-v_i)(t)}_{N+\alpha} \lesssim \int_{t_i}^t \left(\delta_q^{\sfrac{1}{2}} \lambda_q \ell^{-\alpha} \norm{(v_\ell-v_i)(\tau)}_{N+\alpha}+  \lambda_q^{-\gamma} \delta_{q+1} \ell^{-N-1+\alpha} \right) \, d\tau, \quad  \text{if }  \lvert t-t_i \rvert \leq 2 \theta_{q+1} \,, 
\end{equation*}
where we also used again that $\lambda_{q}^{-3\alpha} \leq \ell^{2\alpha}$ by \eqref{l:bound}. Closing a Gr\"onwall exactly as before, we deduce \eqref{e:vi-vlCN}. Inserting this estimate in \eqref{e:pi-plCNaux} and \eqref{N_der_Dtl_v_l-v_i}, we conclude \eqref{e:pi-plCN} and \eqref{e:vi-vlmatCN} as well.
\end{proof}

At difference from \cite{BDLSV2019},  $\overline v_q$ is not purely a gluing of exact solutions $v_i$ to Euler from an initial datum $v_\ell(t_i)\,.$ Instead, we glue the first exact solution $v_0$ to $v_q$ in the interval $I_0$ and the last exact solution $v_n$ to $ v_q$ in the interval $I_{n+1}$. This is necessary in order to guarantee that $\overline v_q(t)= v_q(t)$ outside the new bad set and hence the crucial property \ref{Bq:v}. In addition to Proposition \ref{p:stabilityvi}, we thus need improved estimates (with respect to Proposition \ref{p:mollifaction}) on $v_q-v_\ell\,$ in $I_0 \cup I_{n+1}\,.$ Since $v_q(t) \neq v_\ell(t)$, such estimates can no longer rely on stability estimates via closing a suitable Gr\"onwall inequality, but solely on mollification estimates and the better estimates \eqref{v_q:betterestimate} of $v_q$ on $\widehat{\mathcal{B}}_q^c$, which the inductive assumption \ref{Bq:vii} guarantees.

\begin{proposition}\label{p:betterestvq-vl} For $t \in \widehat{\mathcal{B}}_{q}^c\,,$
we have the estimates 
\begin{align}
\norm{(v_q- v_\ell)(t)}_{N+\alpha} &\lesssim \tau_{q+1} \delta_{q+1} \ell^{-N-1+\alpha} \quad &&\forall N \geq 0 \,, \label{e:vq-vlCNbetter}\\
\norm{\nabla(p_q-p_\ell)(t)}_{N+ \alpha} &\lesssim \lambda_{q}^{-\gamma} \delta_{q+1} \ell^{-N-1+\alpha} \quad &&\forall N \geq 0 \,,\label{e:pq-plCNbetter}\\
\norm{ D_{t,\ell} (v_q-v_\ell)(t)}_{N+\alpha}&\lesssim \lambda_{q}^{-\gamma} \delta_{q+1} \ell^{-N-1+\alpha} \quad &&\forall N \geq 0\,. \label{e:vq-vlmatCNbetter}
\end{align}
where $D_{t, \ell}$ is the material derivative defined in \eqref{e:materialderivative}.
\end{proposition}
\begin{proof} We recall from assumption \ref{Bq:vii}, that $v_q$ satisfies the better estimates \eqref{v_q:betterestimate} on $\widehat{\mathcal{B}}_q^c\,.$ Since $ \ell_{q-1}^{-N} \leq   \ell_{q}^{-N}$ by the definition, we have in particular (with $\ell= \ell_q$ as before)
\begin{equation}\label{e:glueingbetterest}
\norm{ v_q (t)}_{N+1} \leq \delta_{q-1}^{\sfrac{1}{2}} \lambda_{q-1} \ell^{-N} \quad \forall t \in \widehat{\mathcal{B}}_{q}^c \text{ and } \forall N \geq 0 \,.
\end{equation}
We deduce from standard mollification estimates, as in the proof of Proposition \ref{p:mollifaction}, that for $t \in \widehat{\mathcal{B}}_{q}^c$
\begin{equation}\label{e:vqaux}
\norm{ (v_q-v_\ell)(t)}_{N+\alpha} \lesssim \ell^{1-N-\alpha} \delta_{q-1}^{\sfrac{1}{2}} \lambda_{q-1} = \tau_{q+1} \delta_{q+1} \ell^{-N-1+ \alpha} \ell^{-2\alpha} \frac{\delta_{q-1}^{\sfrac{1}{2}}\lambda_{q-1}}{\delta_q^{\sfrac{1}{2}}\lambda_q} \lesssim \tau_{q+1} \delta_{q+1} \ell^{-N-1+ \alpha}
\end{equation}
where  in the last inequality, we used \eqref{l:bound} together with 
\begin{equation}\label{e:vqaux2}
\delta_{q-1}^{\sfrac{1}{2}} \lambda_{q-1} \big(\delta_q^{\sfrac{1}{2}}\lambda_q \big)^{-1}\leq \lambda_q^{-6\alpha} \leq \lambda_q^{-3\alpha}  \,,
\end{equation}
which holds true if we require that $\alpha$ is chosen sufficiently small in order to satisfy
\begin{equation*}
6\alpha b \leq (b-1)(1-\beta) \,.
\end{equation*} 
Observe that $(v_\ell-v_q)$ is divergence-free and that, since $R_q\equiv 0$ on $\T^3 \times \widehat{\mathcal{B}}_{q}^c$ by assumption \ref{Bq:vi}, $v_q$ is an exact solution of Euler on $\T^3 \times \widehat{\mathcal{B}}_{q}^c\,.$ Consequently,  $v_\ell-v_q$ satisfies \eqref{e:vl-vi} (with $v_i$ replaced by $v_q$) and $p_\ell-p_q$ satisfies \eqref{e:pl-pi} (with $p_i$ replaced by $p_q$) on $\T^3 \times \widehat{\mathcal{B}}_{q}^c\,.$ By Proposition \ref{p:CZO_C_alpha}, we deduce, using \eqref{e:Holderproduct}, \eqref{e:vqaux},  \eqref{e:vq-vlCNbetter} and \eqref{e:RlCN}, that for $t \in \widehat{\mathcal{B}}_q^c$
\begin{align*}
&\norm{\nabla (p_q-p_\ell)(t)}_{N+\alpha}\\
&\lesssim \left(\norm{ \nabla v_\ell }_{N+\alpha} + \norm{ \nabla v_q }_{N+\alpha}\right)\norm{v_\ell-v_q}_\alpha + \left(\norm{ \nabla v_\ell }_{\alpha} + \norm{ \nabla v_q }_{\alpha}\right)\norm{v_\ell-v_q}_{N+\alpha} + \norm{\div R_\ell}_{N+\alpha} \\
&\lesssim \delta_{q}^{\sfrac{1}{2}} \lambda_{q} \tau_{q+1} \delta_{q+1} \ell^{-N-1} + \lambda_q^{-\gamma-3\alpha} \delta_{q+1} \ell^{-N-1-\alpha} \\
&\lesssim \lambda_q^{-\gamma} \delta_{q+1} \ell^{-N-1+ \alpha} \,,
\end{align*}
where we used that $\delta_{q-1}^{\sfrac{1}{2}} \lambda_{q-1} \leq \delta_q^{\sfrac{1}{2}}\lambda_q$ and \eqref{l:bound} in the last two inequalities. As for the material derivative, we have, using the equation for $v_\ell-v_q$ as in the proof of Proposition \ref{p:stabilityvi}, by \eqref{e:pq-plCNbetter}, \eqref{e:vqaux}, \eqref{e:vq-vlCNbetter} and \eqref{e:RlCN}
\begin{align*}
\norm{D_{t, \ell} (v_q-v_\ell)}_{N+\alpha} &\lesssim \norm{ \nabla(p_\ell-p_q)}_{N+\alpha} +\norm{\nabla v_q}_{\alpha} \norm{ v_\ell-v_q}_{N+\alpha} + \norm{\nabla v_q}_{N+\alpha} \norm{ v_\ell-v_q}_{\alpha}  + \norm{\div R_\ell}_{N+\alpha} \\
&\lesssim \lambda_q^{-\gamma} \delta_{q+1} \ell^{-N-1-\alpha} + \tau_{q+1} \delta_{q+1} \ell^{-N-1} \delta_{q-1}^{\sfrac{1}{2}} \lambda_{q-1}  +\lambda_q^{-\gamma-3\alpha} \delta_{q+1}  \ell^{-N-\alpha}
\end{align*}
which gives \eqref{e:vq-vlmatCNbetter} by observing $\delta_{q-1}^{\sfrac{1}{2}} \lambda_{q-1} \leq \delta_q^{\sfrac{1}{2}} \lambda_q$ and using \eqref{l:bound}.
\end{proof}

\subsection{Proof of the estimates \eqref{e:pglueing4}--\eqref{e:pglueing6} on $\overline v_q$ }
We show how the estimates \eqref{e:pglueing4}--\eqref{e:pglueing6}, when restricted to $J_q$, are an immediate consequence of Proposition \ref{p:stabilityvi} and \ref{p:betterestvq-vl}. By construction 
\begin{equation*}
\overline v_q - v_q = \sum_{i=0}^n \eta_g^i (v_i-v_q) = \sum_{i=0}^n \eta_g^i (v_i-v_\ell) - \sum_{i=0}^n \eta_g^i (v_q-v_\ell)\,.
\end{equation*}
Since $\supp \eta_g^i \subset \{t \in J_q: \lvert t- t_i \rvert \leq 2 \theta_{q+1} \}$, we can use \eqref{e:vi-vlCN} and \eqref{e:vl-vqC0} to estimate
\begin{align*}
\norm{ \overline v_q-v_q}_{0} \leq   \sup_{i=0, \dots, n} \norm{\eta_g^i(v_i-v_\ell)}_{\alpha} + \norm{ v_\ell- v_q}_0  &\lesssim \tau_{q+1}\delta_{q+1} \ell^{-1+\alpha} + \delta_q^{\sfrac{1}{2}} \lambda_q \ell \lesssim  \delta_{q+1}^{\sfrac{1}{2}}\lambda_q^{-{\sfrac{\gamma}{2}-\sfrac{3\alpha}{2}}} \,,
\end{align*}
which proves \eqref{e:pglueing4}. To prove \eqref{e:pglueing5}, we write 
\begin{equation*}
\overline v_q - v_\ell = \sum_{i=0}^n \eta_g^i (v_i-v_\ell) + (1-\eta_g) (v_q-v_\ell)\,.
\end{equation*}
Since $\supp(1- \eta_g) \subset  \widehat{\mathcal{B}}_q^c$ by construction, we can use \eqref{e:vq-vlCNbetter} together with \eqref{e:vi-vlCN} to estimate
\begin{align}\label{e:final}
\norm{\overline v_q- v_\ell}_{0} &\leq \sup_{i=0, \dots, n} \norm{ \eta_g^i (v_i-v_\ell)}_{\alpha} + \norm{(1-\eta_g)(v_q-v_\ell)}_{ \alpha} \lesssim \tau_{q+1} \delta_{q+1} \ell^{-1 +\alpha} \lesssim  \delta_q^{\sfrac{1}{2}} \lambda_q \ell
\end{align}
and
\begin{align*}
\norm{\overline v_q- v_\ell}_{N+1} &\leq \sup_{i=0, \dots, n} \norm{ \eta_g^i (v_i-v_\ell)}_{N+1+\alpha} + \norm{(1-\eta_g)(v_q-v_\ell)}_{N+1+ \alpha} \\
&\lesssim \tau_{q+1} \delta_{q+1} \ell^{-N-2 +\alpha} \lesssim \delta_q^{\sfrac{1}{2}} \lambda_q \ell^{-N} \quad \forall N \geq 0 \,,
\end{align*}
proving \eqref{e:pglueing5}. Finally, \eqref{e:pglueing6} follows immediately combining the former estimate with \eqref{e:vlCN}\,.

\subsection{Estimates on the vector potentials} To improve the estimates on the Reynolds stress $\overline R_q$, it is useful to consider the vector potentials associated to $v_i\,,$ $v_\ell$ and $v_q$ defined by
\begin{align*}
z_i &=\mathcal{B} v_i := (-\Delta)^{-1}  \curl v_i  \quad i=0, \dots, n  \,, \\
z_\ell &:= \mathcal{B} v_\ell \,, \\
z_q&:= \mathcal{B} v_q \,,
\end{align*}
where $\mathcal B$ is the Bio-Savart operator. By construction, 
$\div z_i= \div z_\ell= \div v_q=0$ and 
$$\curl z_i= (-\Delta)^{-1} \curl \curl v_i = v_i \quad \curl z_\ell = v_\ell \quad \curl z_q = v_q\,,$$
since $v_i\,,$ $v_\ell$ and $v_q$ are divergence free. Thus, we view $z_i-z_\ell$ (and $z_q - v_\ell$) as potential of first order of $v_i-v_\ell$ (and $v_q-v_\ell$) and as such, we expect the stability estimates $z_i- z_\ell$ (and $z_q-z_\ell$) to improve by a factor of $\ell$. We make this heuristic rigorous in the following 

\begin{proposition}\label{p:vectorpotentials} For $\lvert t-t_i \rvert \leq 2 \theta_{q+1}$ 
\begin{align}
\norm{(z_i - z_\ell)(t)}_{N+\alpha} &\lesssim \tau_{q+1} \delta_{q+1} \ell^{-N+\alpha} \quad &&\forall N \geq 0 \label{e:zi-zlCN}\\
\norm{D_{t, \ell}(z_i-z_\ell) (t)}_{N+ \alpha} &\lesssim \lambda_q^{-\gamma} \delta_{q+1} \ell^{-N+\alpha}  \quad &&\forall N \geq 0 \label{e:zi-zlmatCN}\,,
\end{align}
where $D_{t, \ell}$ denotes the material derivative as defined in \eqref{e:materialderivative}.
\end{proposition}

The proof of Proposition \ref{p:vectorpotentials} follows closely \cite{BDLSV2019}. The next proposition on the other hand, should be seen as the analogue of Proposition \ref{p:betterestvq-vl} and exploits crucially that $v_q$ is an exact solution of Euler on $\T^3 \times \widehat{\mathcal{B}}_q^c$ with better estimates.

\begin{proposition}\label{p:vectorpotentials2} For $t \in \widehat{ \mathcal{B}}_q^c$ we have
\begin{align}
\norm{(z_q-z_\ell)(t)}_{N+ \alpha} &\lesssim \tau_{q+1} \delta_{q+1} \ell^{-N+ \alpha} \quad &&\forall N \geq 0, \label{e:zq-zlCN}\\
\norm{ D_{t, \ell} (z_q- z_\ell)(t)}_{N+\alpha} &\lesssim \lambda_q^{-\gamma} \delta_{q+1} \ell^{-N+\alpha} \quad &&\forall N \geq 0 \label{e:zq-zlmatCN}\,.
\end{align}
where $D_{t, \ell}$ denotes the material derivative as defined in \eqref{e:materialderivative}.
\end{proposition}

\begin{proof}[Proof of Proposition \ref{p:vectorpotentials}] We set $\tilde z_i := z_\ell-z_i$. Recall that $(v_\ell-v_i)$ solves \eqref{e:vl-vi}, so that 
\begin{equation}
\partial_t \curl \tilde z_i + (v_\ell \cdot \nabla) \curl \tilde z_i = - \nabla (p_\ell-p_i) - ( \curl \tilde z_i \cdot \nabla) v_i + \div  R_\ell \,.
\end{equation}
We rewrite, using $\div \tilde z_i= \div v_\ell=0$, 
\begin{align*}
[(v_\ell \cdot \nabla ) \curl \tilde z_i]^j &= \partial_k\left( v_\ell^k \left [\curl \tilde z_i\right]^j \right) = [\curl((v_\ell \cdot \nabla) \tilde z_i )]^j   + \partial_k \left( [\tilde z_i \times \nabla  v_\ell^k]^j\right),  \\
[( \curl \tilde z_i \cdot \nabla) v_i]^j &=\partial_k \left(\left[ \curl \tilde z_i\right]^k v_i^j\right) = \div \curl \left(\tilde z_i v_i^j\right) + \partial_k\left( [ \tilde z_i \times \nabla v_i^j ]^k\right) = \partial_k\left( [ \tilde z_i \times \nabla v_i^j ]^k\right),
\end{align*} 
where we used the convention to sum over repeated indices. Setting 
\begin{equation*}
[(z \times \nabla) v]^{jk}= [z \times \nabla v^k ]^j = \epsilon_{jlm} z^l \partial_m v^k\,,
\end{equation*}
where $ \epsilon_{jlm}$ denotes the Levi-Civita symbol, we obtain that 
\begin{equation*}
\curl (\partial_t \tilde z_i + (v_\ell \cdot \nabla) \tilde z_i) = -\div \left((\tilde z_i \times \nabla ) v_\ell +[(\tilde z_i \times \nabla) v_i]^T \right) - \nabla (p_\ell - p_i) + \div R_\ell \,.
\end{equation*}
Taking the $\curl$ of the above equation and recalling that $(- \Delta)= \curl \curl + \nabla \div$, we find 
\begin{equation*}\label{e:tildezi}
(-\Delta)(\partial_t \tilde z_i  + (v_\ell \cdot \nabla) \tilde z_i) = F \,,
\end{equation*}
where 
\begin{equation}\label{def:F}
F= - \nabla \div((\tilde z_i \cdot \nabla) v_\ell)- \curl \div\left((\tilde z_i \times \nabla) v_\ell + [(\tilde z_i \times \nabla) v_i]^T\right) + \curl \div R_\ell \,.
\end{equation}
We deduce from Proposition \ref{p:CZO_C_alpha} and \eqref{e:Holderproduct} that 
\begin{align}
\norm{ D_{t, \ell} \tilde z_i}_{N+\alpha} &\lesssim \norm{ (\tilde z_i \cdot \nabla ) v_\ell}_{N+\alpha} + \norm{( \tilde z_i \times \nabla) v_\ell}_{N+\alpha}+ \norm{ \left[(\tilde z_i \times \nabla) v_i\right]^T}_{N+\alpha} + \norm{R_\ell}_{N+\alpha} \nonumber \\
&\leq (\norm{ v_i}_{N+1 + \alpha} + \norm{v_\ell}_{N+1+\alpha}) \norm{\tilde z_i}_\alpha +  (\norm{ v_i}_{1 + \alpha} + \norm{v_\ell}_{1+\alpha}) \norm{\tilde z_i}_{N+\alpha}+ \norm{R_\ell}_{N+\alpha} \nonumber \\
&\lesssim \delta_q^{\sfrac{1}{2}} \lambda_q \ell^{-N-\alpha} \norm{\tilde z_i}_\alpha +  \delta_q^{\sfrac{1}{2}} \lambda_q \ell^{-\alpha} \norm{\tilde z_i}_{N+\alpha} +\lambda_q^{-\gamma - 3\alpha} \delta_{q+1}  \ell^{-N-\alpha} \,, \label{e:vectorpotentialaux1}
\end{align}
where the last inequality is a consequence of \eqref{e:vlCN}, \eqref{e:viCN} and \eqref{e:RlCN}. In particular for $N=0$, we deduce from Proposition \ref{p:transport} that, since $\tilde z_i(t_i)=0$,
\begin{equation*}
\norm{ \tilde z_i (t)}_\alpha \lesssim \int_{t_i}^t \left(\delta_q^{\sfrac{1}{2}} \lambda_q \ell^{-\alpha} \norm{\tilde z_i (\tau)}_\alpha +\lambda_q^{-\gamma - 3\alpha} \delta_{q+1}  \ell^{-\alpha} \right) \, d \tau, \quad \text{if } \lvert t-t_i \rvert \leq 2 \theta_{q+1}\,.
\end{equation*}
Since $\delta_q^{\sfrac{1}{2}} \lambda_q \ell^{\alpha} \theta_{q+1} \leq \ell^\alpha \leq 1$ and that $\lambda_q^{-3\alpha} \ell^{-\alpha} \leq \ell^{\alpha}$ by \eqref{l:bound}, by Gr\"onwall's inequality we get 
\begin{equation*}
\norm{ \tilde z_i (t)}_\alpha \lesssim \theta_{q+1} \lambda_q^{-\gamma} \delta_{q+1} \ell^\alpha = \tau_{q+1} \delta_{q+1} \ell^\alpha,  \quad \text{if } \lvert t-t_i \rvert \leq 2 \theta_{q+1} \, .
\end{equation*}
Inserting this bound back in \eqref{e:vectorpotentialaux1}, we also obtain \eqref{e:zi-zlmatCN} for $N=0\,.$ As for the higher derivatives, we simply observe that the operator $\nabla \mathcal{B}$ is bounded on H\"older spaces by Proposition \ref{p:CZO_C_alpha} and hence for $N \geq 1\,,$ we deduce from \eqref{e:vi-vlCN}
\begin{equation*}
\norm{ \tilde z_i}_{N+ \alpha} = \norm{ \nabla \tilde z_i}_{N-1+\alpha} = \norm{ \nabla \mathcal{B} (v_i-v_\ell)}_{N-1-\alpha} \lesssim \tau_{q+1} \delta_{q+1} \ell^{-N+\alpha} \,.
\end{equation*}
The estimate \eqref{e:zi-zlmatCN} for $N\geq 1$ is obtained by writing $\partial^\theta D_{t, \ell} \tilde z_i =  D_{t, \ell} \partial^\theta \tilde z_i +  [\partial^\theta, (v_\ell \cdot \nabla) ] \tilde z_i$ for a derivative $\partial^\theta$ with $\lvert\theta \rvert =N$, estimating separately the resulting commutator as in the proof of Proposition \ref{p:stabilityvi}.
\end{proof}

\begin{proof}[Proof of Proposition \ref{p:vectorpotentials2}] Observe that the operator $\mathcal{B}$ commutes with convolution, hence $z_\ell = z_q \ast \varphi_\ell$. Moreover, as a consequence of assumption \ref{Bq:vii} and the fact that $\ell_{q} \leq \ell_{q-1}$, we have on $\widehat{\mathcal{B}}_q^c$ (for $\alpha$ small enough) the better estimates \eqref{e:glueingbetterest}--\eqref{e:vqaux}.
Since $(-\Delta) z_q = \curl v_q$, standard estimates for the Laplace equation give $\norm{z_q}_2\leq \|z_q\|_{2+\alpha} \lesssim \norm{v_q}_{1+\alpha}$, which together with \eqref{improved_moll}, implies
\begin{equation*}
\norm{ (z_q-z_\ell)(t)}_\alpha \lesssim  \ell^{2-\alpha} \norm{ z_q(t)}_{2} \lesssim \ell^{2-\alpha} \delta_{q-1}^{\sfrac{1}{2}} \lambda_{q-1}\ell^{-\alpha} =  \tau_{q+1} \delta_{q+1} \ell^{-2\alpha} \frac{\delta_{q-1}^{\sfrac{1}{2}}\lambda_{q-1}}{\delta_q^{\sfrac{1}{2}} \lambda_q} \lesssim \tau_{q+1} \delta_{q+1} \ell^\alpha  \quad t \in\widehat{\mathcal{B}}_q^c \,,
\end{equation*}
where in the last inequality we reused \eqref{e:vqaux2}.
As for derivatives of higher order, we recall that $\nabla \mathcal{B}$ is bounded on H\"older spaces by Proposition \ref{p:CZO_C_alpha}. We then estimate using \eqref{e:vqaux} for $t \in \widehat{\mathcal{B}}_q^c$
\begin{align*}
\norm{(z_\ell- z_q)(t)}_{N+1+\alpha} &= \norm{\nabla \mathcal{B} (v_\ell- v_q)(t)}_{N+\alpha} \lesssim \norm{(v_\ell- v_q)(t)}_{N+\alpha}  \lesssim \tau_{q+1} \delta_{q+1}  \ell^{-N-1+\alpha} \,,
\end{align*}
where the last inequality uses again $\alpha$ small enough as in  \eqref{e:vqaux2}.
As for the material derivative, we observe that by assumption \ref{Bq:vi}, $v_q$ is a smooth solution of Euler on $ \T^3 \times \widehat{\mathcal{B}}_q^c\,.$ Hence we can argue as in the proof of Proposition \ref{p:vectorpotentials} to obtain, for $t \in\widehat{\mathcal{B}}_q^c$
\begin{align*}
&\norm{ D_{t, \ell} (z_\ell- z_q)}_{N+\alpha} \\
&\lesssim \left(\norm{ v_\ell }_{N+1+\alpha}+ \norm{ v_q}_{N+1+\alpha}\right)\norm{z_\ell- z_q}_{\alpha} + \left(\norm{ v_\ell }_{1+\alpha}+ \norm{ v_q}_{1+\alpha}\right)\norm{z_\ell- z_q}_{N+\alpha} + \norm{R_\ell}_{N+\alpha} \\
&\lesssim  \lambda_q^{-\gamma-3\alpha} \delta_{q+1} \ell^{-N-\alpha},
\end{align*}
where the last inequality follows from combining the estimates \eqref{e:glueingbetterest}, \eqref{e:vlCN}, \eqref{e:zq-zlCN} and \eqref{e:RlCN}. We conclude \eqref{e:zq-zlmatCN} recalling that $\lambda_q^{-3\alpha} \leq \ell^{2\alpha}$ by \eqref{l:bound}.
\end{proof}

\subsection{Proof of the estimates \eqref{e:pglueing7}--\eqref{e:pglueing8} on $\overline R_q$} Since $\supp \overline R_q \subset \T^3 \times \bigcup_{i=0}^{n+1} I_i$ by construction, it suffices to prove both estimates on every interval $I_i$. As in \cite{BDLSV2019}, we will repeatedly use that $\mathcal{R} \curl$ is a bounded operator on H\"older spaces by Proposition \ref{p:CZO_C_alpha} and thereby we improve the estimates on  terms of the form $\mathcal{R}(v_i-v_{i-1}) = \mathcal{R} \curl(z_i-z_{i-1})$  by passing to the vector potentials.

Consider now first the case $i \in \{1, \dots, n \}\,.$  Recall from \eqref{def:barRq} 
\begin{equation*}
\overline R_q = \partial_t \eta_g^i  \mathcal{R}(v_i-v_{i-1}) - \eta_g^i (1-\eta_g^i) ((v_i-v_{i-1}) \otimes (v_i-v_{i-1})) \quad \text{ on } I_i \,.
\end{equation*}
Recall from Lemma \ref{l:cutoffglueing} that $\norm{\partial^N_t \eta_g^i}_0 \lesssim \tau_{q+1}^{-N}$ and that $\supp \eta_g^i \subset  \{ t : \lvert t-t_i \rvert \leq 2 \theta_{q+1} \}\,.$ Therefore, we bound, using also \eqref{e:Holderproduct}, \eqref{e:vi-vlCN} and \eqref{e:zi-zlCN},
\begin{align}
\norm{ \overline R_q (t)}_{N+\alpha} &\lesssim \tau_{q+1}^{-1} \norm{ \mathcal{R} \curl (z_i-z_{i-1})}_{N+\alpha} + \norm{ v_i-v_{i-1}}_{N+\alpha} \norm{v_i-v_{i-1}}_\alpha \nonumber \\
&\lesssim \tau_{q+1}^{-1} \norm{z_i-z_{i-1}}_{N+\alpha} + \tau_{q+1}^2 \delta_{q+1}^2 \ell^{-N-2 + 2\alpha} \nonumber \\
&\lesssim \delta_{q+1} \ell^{-N+\alpha} \label{e:final3}\,,
\end{align}
which gives \eqref{e:pglueing7} on $I_i\,.$ As for estimate \eqref{e:pglueing8},  we begin by writing
\begin{align}\label{e:final2}
\norm{ (\partial_t + \overline v_q \cdot \nabla)\overline{R}_q}_{N+\alpha} \leq \norm{ (v_\ell- \overline v_q) \cdot \nabla  \overline R_q}_{N+\alpha} + \norm{ D_{t, \ell} \overline R_q}_{N+\alpha}
\end{align}
where $D_{t, \ell}$ is defined in \eqref{e:materialderivative}. Using \eqref{e:pglueing4}, \eqref{e:final} and \eqref{e:final3}, we estimate the first term on $I_i$
\begin{equation}\label{e:final4}
\norm{(v_\ell-\overline v_q)\cdot \nabla\overline R_q}_{N+\alpha} \lesssim \norm{ v_\ell - \overline v_q}_{N+\alpha} \norm{ \nabla \overline R_q}_0 + \norm{v_\ell- \overline v_q}_0 \norm{\nabla \overline R_q}_{N+\alpha} \lesssim \delta_{q+1}  \delta_q^{\sfrac{1}{2}} \lambda_q \ell^{-N} \,,
\end{equation}
which is better than the desired bound in \eqref{e:pglueing8}. We are left to bound $\norm{D_{t, \ell} \overline R_q}_{N+\alpha}$ on $I_i\,.$ We compute (always on $I_i$)
\begin{align*}
D_{t, \ell} \overline R_q &= \partial_{tt} \eta_g^i \mathcal{R}(v_i-v_{i-1}) + \partial_t \eta_g^i \Big(\partial_t \mathcal{R}(v_i-v_{i-1}) +( v_\ell \cdot \nabla) \mathcal{R}(v_i-v_{i-1})\Big) \\
&- \partial_t \big( \eta_g^i (1-\eta_g^i) \big) \big((v_i-v_{i-1}) \otimes (v_i-v_{i-1}) \big)\\
& - \eta_g^i (1-\eta_g^i) \Big( \big(D_{t, \ell} (v_i-v_{i-1}) \big) \otimes (v_i-v_{i-1})+ (v_i-v_{i-1}) \otimes \big(D_{t, \ell} (v_i-v_{i-1}) \big)\Big)\,.
\end{align*}
We rewrite
\begin{equation*}
\partial_t \eta_g^i \Big(\partial_t \mathcal{R}(v_i-v_{i-1}) +( v_\ell \cdot \nabla) \mathcal{R}(v_i-v_{i-1})\Big)  = \partial_t \eta_g^i \Big((\mathcal{R}\curl) D_{t, \ell} (z_i-z_{i-1}) + [(v_\ell \cdot \nabla), \mathcal{R}\curl ] (z_i-z_{i-1})\Big) \,,
\end{equation*}
where  $[(v_\ell \cdot \nabla), \mathcal{R}\curl ]$ denotes the commutator involving the singular integral operator $\mathcal{R} \curl$. From Proposition \ref{p:commutator}, \eqref{e:vlCN} and \eqref{e:zi-zlCN}, we have 
\begin{align}
\norm{[(v_\ell \cdot \nabla), \mathcal{R}\curl ] (z_i-z_{i-1})}_{N+\alpha} &\lesssim \norm{v_\ell}_{1+\alpha} \norm{ z_i-z_{i-1}}_{N+\alpha}+ \norm{ v_\ell}_{N+1+\alpha} \norm{ z_i-z_{i-1}}_\alpha \nonumber\\
&\lesssim \tau_{q+1} \delta_{q+1} \delta_q^{\sfrac{1}{2}} \lambda_q \ell^{-N}  \,. \label{e:barRqaux}
\end{align}
We can thus estimate (always on $I_i$) using \eqref{e:Holderproduct} on products together with \eqref{e:zi-zlCN}, \eqref{e:zi-zlmatCN}, \eqref{e:barRqaux}, \eqref{e:vi-vlCN} and \eqref{e:vi-vlmatCN}
\begin{align*}
\norm{D_{t, \ell} \overline R_q}_{N+\alpha} &\lesssim \tau_{q+1}^{-2} \norm{z_i-z_{i-1}}_{N+\alpha}  + \tau_{q+1}^{-1}\big( \norm{D_{t, \ell} (z_i-z_{i-1})}_{N+ \alpha} + \norm{[(v_\ell \cdot \nabla), \mathcal{R}\curl ] (z_i-z_{i-1})}_{N+\alpha}\big) \\
&+ \tau_{q+1}^{-1} \norm{ v_i-v_{i-1}}_{N+\alpha} \norm{v_i-v_{i-1}}_\alpha \\
&+ \norm{D_{t, \ell}(v_i-v_{i-1})}_{N+\alpha} \norm{v_i-v_{i-1}}_\alpha +\norm{D_{t, \ell}(v_i-v_{i-1})}_{\alpha} \norm{v_i-v_{i-1}}_{N+\alpha} \\
&\lesssim \tau_{q+1}^{-1} \delta_{q+1} \ell^{-N+\alpha} + \delta_{q+1}\delta_q^{\sfrac{1}{2}} \lambda_q \ell^{-N} + \tau_{q+1} \delta_{q+1}^2 \ell^{-N-2+2\alpha} \\
&\lesssim \tau_{q+1}^{-1} \delta_{q+1} \ell^{-N+\alpha} \,,
\end{align*}
where we used in the last inequality that $\tau_{q+1} \delta_{q}^{\sfrac{1}{2}} \lambda_q \leq \lambda_q^{-3\alpha} \leq \ell^\alpha$ by \eqref{l:bound} and  $\left(\tau_{q+1} \ell^{-1}\right)^2 \leq 1\,.$ This proves \eqref{e:pglueing8} on $I_i$ recalling \eqref{l:bound}.

Finally, let us prove the estimates \eqref{e:pglueing7}--\eqref{e:pglueing8} on $I_0$ and $I_{n+1}\,.$ Recall from \eqref{def:barRq} 
\begin{equation*}
\overline R_q = \partial_t \eta_g^0 \mathcal{R}(v_0-v_q) - \eta_g^0 (1-\eta_g^0) ((v_0-v_q) \otimes (v_0-v_q)) \quad \text{ on } I_0 \,.
\end{equation*}
Arguing as before and writing $z_0-z_q=(z_0- z_\ell)+ (z_\ell-z_q)$ and $v_0-v_q= v_0-v_\ell+ v_\ell-v_q$, we have for $t \in I_0 \subset \widehat{\mathcal{B}}_q^c$ using \eqref{e:zi-zlCN}, \eqref{e:zq-zlCN}, \eqref{e:vi-vlCN} and \eqref{e:vq-vlCNbetter}
\begin{align*}
\norm{ \overline R_q (t)}_{N+\alpha} &\lesssim \tau_{q+1}^{-1} \norm{( z_0- z_q)(t)}_{N+\alpha} + \norm{(v_0- v_q)(t)}_{N+\alpha} \norm{(v_0-v_q)(t)}_\alpha \\
&\lesssim \delta_{q+1} \ell^{-N+\alpha}+  \tau_{q+1}^2 \delta_{q+1}^2 \ell^{-N-2+2\alpha} \\
&\lesssim \delta_{q+1} \ell^{-N+\alpha} \,.
\end{align*}
As for estimate \eqref{e:pglueing8}, we argue as in \eqref{e:final2} and \eqref{e:final4} to reduce ourselves to bound $\norm{D_{t, \ell} \overline{R}_q}_{N+\alpha}.$ Proceeding as before, we obtain for $t \in I_0\subset \widehat{\mathcal{B}}_q^c$ that 
\begin{align*}
\norm{D_{t, \ell} \overline R_q (t) }_{N+\alpha} &\lesssim \tau_{q+1}^{-2} \norm{z_0-z_q}_{N+\alpha}  + \tau_{q+1}^{-1}\big( \norm{D_{t, \ell} (z_0-z_q)}_{N+ \alpha} + \norm{[(v_\ell \cdot \nabla), \mathcal{R}\curl ] (z_0-z_{q})}_{N+\alpha}\big) \\
&+ \tau_{q+1}^{-1} \norm{ v_0-v_{q}}_{N+\alpha} \norm{v_0-v_{q}}_\alpha \\
&+ \norm{D_{t, \ell}(v_0-v_{q})}_{N+\alpha} \norm{v_0-v_{q}}_\alpha +\norm{D_{t, \ell}(v_0-v_{q})}_{\alpha} \norm{v_0-v_{q}}_{N+\alpha} \\
&\lesssim \tau_{q+1}^{-1} \delta_{q+1} \ell^{-N+ \alpha} + \delta_{q+1} \delta_q^{\sfrac{1}{2}} \lambda_q \ell^{-N} +\tau_{q+1} \delta_{q+1}^2 \ell^{-N-2+2\alpha}\,,
\end{align*}
where we used Proposition \ref{p:commutator} to estimate the commutator as well as the estimates \eqref{e:zq-zlCN}, \eqref{e:zq-zlmatCN}, \eqref{e:vlCN}, \eqref{e:vq-vlCNbetter} and \eqref{e:vq-vlmatCNbetter}. We conclude \eqref{e:pglueing8} on $I_0\,.$ The estimates on $I_{n+1}$ follow in the same way up to exchanging the role $(v_0, \eta_g^0)$ with $(v_n, \eta_g^n)\,.$ This concludes the proof of Proposition \ref{p:glueing}.

\section{Perturbation}\label{s:perturbation}
In this section we will construct the perturbation $w_{q+1}$ and consequently define
\begin{equation}\label{e:final5}
v_{q+1}:=\overline v_q +w_{q+1}\,,
\end{equation}
where $(\overline v_q, \overline R_q)$ is a smooth solution of \eqref{ER} as given by Proposition \ref{p:glueing}. Following the construction of \cite{BDLSV2019}, the perturbation will be highly oscillatory and it will be based on the Mikado flows. As for the gluing step, also here there will be some changes with respect to \cite{BDLSV2019}. For instance, the fact that we are not interested in prescribing an energy profile, allows us to simplify the choice of the amplitude of the perturbation. For this reason, we will give a complete proof of all the estimates.

Let now $\mathcal{B}_{q+1} \subset [0,T]$ be the bad set belonging to $(\overline{v}_q, \overline R_q)$ (see  Proposition \ref{p:glueing}). Note in particular that by Proposition \ref{p:glueing}, $\mathcal{B}_{q+1}$ already satisfies the size properties \ref{Bq:i}--\ref{Bq:iv} at step $q+1$ and we will leave the bad set $\mathcal{B}_{q+1}$ unchanged. Thus, to prove Proposition \ref{p:perturbation}, we are left only to check the two estimates \eqref{v_q+1-bar v_q} and \eqref{R_q+1} as well as the properties \ref{Bq:v}--\ref{Bq:vii} (with $q$ replaced by $q+1$). Since by Proposition \ref{p:glueing}, the couple $(\bar v_q, \overline{R}_q)$ already satisfies the more restrictive properties \eqref{e:pglueing1}, \eqref{e:pglueing2} and \eqref{e:pglueing6}, the properties \ref{Bq:iv}--\ref{Bq:vii} can be achieved by ensuring that the temporal support of $w_{q+1}$ is contained in a $\tau_{q+1}$ neighbourhood of the time support of $\overline R_q$. In particular, this will ensure that $\supp R_{q+1}\subset \T^3\times \left\{t\in \mathcal{B}_{q+1} \, : \, \dist(t,\mathcal{G}_{q+1})> \tau_{q+1}  \right\}$, or in other words, that the new Reynolds stress $R_{q+1}$ is localized in the new real bad set $\widehat{\mathcal{B}}_{q+1}$ which is made of disjoint intervals of length $3\tau_{q+1}$.

A crucial relation that will allow us to close the estimates on $R_{q+1}$ will be
\begin{equation}
\label{ell_less_lambda_q+1}
\ell^{-1}\ll \lambda_{q+1},
\end{equation}
that is a consequence of $\gamma+3\alpha<2(b-1)(1-\beta)$. By our bound on $\gamma$ in \eqref{gamma:bound2} (actually \eqref{gamma:bound1} would suffice here), the latter holds if $\alpha$ is sufficiently small.

\subsection{Mikado flows}

We now recall the construction of Mikado flows used in  \cite{BDLSV2019}.

\begin{lemma}\label{l:Mikado}For any compact subset $\mathcal N\subset\subset \S^{3\times3}_+$
there exists a smooth vector field 
$$
W:\mathcal N\times \T^3 \to \R^3, 
$$
such that, for every $R\in\mathcal N$ 
\begin{equation}\label{e:Mikado}
\left\{\begin{aligned}
\div_\xi(W(R,\xi)\otimes W(R,\xi))&=0 \\ \\
\div_\xi W(R,\xi)&=0,
\end{aligned}\right.
\end{equation}
and
\begin{eqnarray}
	\fint_{\T^3} W(R,\xi)\,d\xi&=&0,\label{e:MikadoW}\\
    \fint_{\T^3} W(R,\xi)\otimes W(R,\xi)\,d\xi&=&R.\label{e:MikadoWW}
\end{eqnarray}
\end{lemma}

Using the fact that $W(R,\xi)$ is $\T^3-$periodic and has zero mean in $\xi$, we write
\begin{equation}\label{e:Mikado_Fourier}
W(R,\xi)=\sum_{k \in \Z^3\setminus\{0\}}  a_k(R) A_k e^{ik\cdot \xi}
\end{equation}
for some  smooth functions $R\mapsto a_k(R)$ and complex vectors $A_k\in \mathbb{C}^3$ satisfying $A_k  \cdot k=0$ and $|A_k|=1$. From the smoothness of $W$, we further infer
\begin{equation}\label{e:a_k_est}
\sup_{R\in \mathcal N}\abs{D^N_R a_k(R)}\leq  \frac{C(\mathcal{N},N,m)}{\abs {k}^m}
\end{equation}
for some constant $C$, which depends, as highlighted in the statement, on $\mathcal{N}$, $N$ and $m$.
\begin{remark}\label{r:choice_of_M}
Later in the proof, the estimates \eqref{e:a_k_est} will be used with a specific choice of the compact set $\mathcal{N}$ and of the integers $N$ and $m$: this specific choice will then determine the universal constant $M$ appearing in Proposition \ref{p:iterativeprop}.
\end{remark}

Using the Fourier representation, we see that from \eqref{e:MikadoWW}
\begin{equation}\label{e:Mikado_stationarity}
W(R,\xi)\otimes W(R,\xi) = R+\sum_{k\neq 0} C_{k}(R) e^{i k\cdot \xi}
\end{equation}
where
\begin{equation}\label{e:Ck_ind}
C_k  k=0 \quad \mbox{and} \quad
\sup_{R\in \mathcal N}\abs{D^N_R C_k(R)}\leq \frac{C (\mathcal{N}, N, m)}{\abs {k}^m}
\end{equation}
for any $m,N \in \N$. It will also be useful to write the Mikado flows in terms of a potential as follows
\begin{align}
\curl_{\xi}\left(\left(\frac{ik\times  A_k}{\abs{k}^2}\right) e^{i k\cdot \xi}\right) &= -i\left(\frac{ik\times  A_k}{\abs{k}^2}\right)\times k  e^{i k\cdot \xi} 
= -\frac{k\times (k\times  A_k)}{\abs{k}^2}  e^{i k\cdot \xi} =  A_k  e^{i k\cdot \xi}. \label{e:Mikado_Potential}
\end{align}

\subsection{The stress tensor $\tilde R_{q,i}$} Recall that $\overline{R}_q$ is supported in the set $\T^3\times I$, where $I$ is the union of disjoint intervals of length $\tau_{q+1}$. Thus we can write
$$
I=\bigcup_{i}I_i, \, \text{ where } |I_i|=\tau_{q+1}.
$$
The following lemma gives the family of cutoffs that will allow us to localize the perturbation (and thus the new Raynolds stress) in the new real bad set $\widehat{\mathcal{B}}_{q+1}$.

\begin{lemma}\label{l:cutoff_pert}
There exist smooth cutoff functions $\{\eta_p^i\}_i$ such that $\eta_p^i\big|_{I_i}\equiv 1$, $ \supp \eta_p^i \cap \supp \eta_p^j=\emptyset$ if $i\neq j$, $\supp \eta_p^i \subset \left\{t\in I_i\, : \, \dist(t,I_i)< \tau_{q+1} \right\}$. Moreover,  for any $i$ and $N\geq 0$ we have
\begin{equation}\label{est_cutoff_pert}
\|\partial_t^N\eta_p^i\|_{0}\lesssim \tau_{q+1}^{-N}.
\end{equation}
\end{lemma}
Let $s_i$ be the middle point of $I_i$. Define the flows $\Phi_i$ associated to the velocity field $\overline v_q$ as the solution of 
$$
\left\{\begin{array}{l}
(\partial_t +\overline v_q\cdot \nabla)  \Phi_i =0\\
\Phi_i(x,s_i)=x.
\end{array}\right.
$$
Define also
\begin{equation}\label{d:tildeRqi}
\tilde R_{q,i}:=\frac{\nabla \Phi_i \left(\delta_{q+1} \Id -\overline R_q\right) \nabla\Phi_i^T}{\delta_{q+1}}.
\end{equation}
We have the following
\begin{lemma}\label{l:Rqi_identity}
For $a\gg 1$ sufficiently large, we have 
\begin{equation}\label{phi_close_id}
\|\nabla \Phi_i(t)-\Id\|_0 \leq \frac{1}{2}, \quad \forall t\in \supp \eta_p^i.
\end{equation}
Moreover, for all $(x,t)\in \T^3\times \supp \eta_p^i$
$$
\tilde R_{q,i}(x,t)\in B_{\frac12}(\Id)\subset\mathcal{S}^{3\times 3}_+,
$$
where $ B_{\frac12}(\Id) $ denotes the ball of radius $\frac12$ around the identity, in the space of positive definite matrices.
\end{lemma}
\begin{proof}
By applying \eqref{e:pglueing6} and \eqref{e:Dphi_near_id} we obtain 
\begin{equation}\label{gradphi_close_id}
\|\nabla \Phi_i(t)-\Id\|_0 \lesssim |t-s_i|\|\overline v_q\|_1\lesssim \tau_{q+1}\delta_q^{\sfrac12 }\lambda_q\leq \lambda_q^{-\gamma}.
\end{equation}
Furthermore, by definition we have 
\begin{align*}
\tilde R_{q,i}-\Id&= - \nabla \Phi_i \frac{\overline R_q}{\delta_{q+1}} \nabla \Phi_i^T+\nabla \Phi_i \nabla\Phi_i^T-\Id\\
&= - \nabla \Phi_i \frac{\overline R_q}{\delta_{q+1}} \nabla \Phi_i^T+\left( \nabla\Phi_i -\Id\right) \nabla\Phi_i^T+ \left(\nabla\Phi_i -\Id \right)^T,
\end{align*}
from which, by using \eqref{e:pglueing7} and \eqref{gradphi_close_id}, we obtain for $t \in \supp \eta_p^i$
$$
\| \tilde R_{q,i}-\Id\|_0\lesssim \frac{\|\overline R_q\|_0}{\delta_{q+1}}+ \|\nabla\Phi_i -\Id\|_0\lesssim \ell^\alpha+\lambda^{-\gamma}.
$$
By choosing $a\gg1$ large enough, we conclude $ \tilde R_{q,i}(x,t)\in B_{\frac12}(\Id)$ for every $(x,t)\in \T^3\times \supp \eta_p^i$.
\end{proof}

\subsection{The perturbation,  the constant M and the properties \ref{Bq:v} and \ref{Bq:vii}}
We define the principal part the the perturbation as 
$$
w_o:= \sum_i \eta_p^i \delta_{q+1}^{\sfrac12} \nabla \Phi_i^{-1} W(\tilde R_{q,i}, \lambda_{q+1}\Phi_i)=\sum_i w_{o,i},
$$
where Lemma \ref{l:Mikado} is applied with $\mathcal{N}=\overline B_{\frac12}(\Id)$. Notice that from Lemma \ref{l:Rqi_identity} it follows that $W(\tilde R_{q,i}, \lambda_{q+1}\Phi_i)$ is well defined. Using the Fourier series representation \eqref{e:Mikado_Fourier} we obtain
$$
w_{o,i}=\sum_{k\neq 0}\eta_p^i \delta_{q+1}^{\sfrac12} a_k(\tilde R_{q,i})\nabla \Phi_i^{-1}A_k e^{i\lambda_{q+1} k\cdot \Phi_i}.
$$
The choice of $w_o$ is motivated by the fact that the vector fields $U_{i,k}=\nabla \Phi_i^{-1} A_k e^{i\lambda_{q+1} k\cdot \Phi_i}$ solve 
\begin{equation}\label{rel_div_free}
(\partial_t+\overline v_q\cdot \nabla)U_{i,k}=\nabla \overline v_q^T U_{i,k}\,.
\end{equation}
In particular, since $\div U_{i, k}(x, s_i)= 0$ for all $x \in \T^3$, $U_{i,k}$ remains divergence free. 

For notational convenience we set 
$$
b_{i,k}(x,t):=\eta_p^i(t)\delta_{q+1}^{\sfrac12} a_k(\tilde R_{q,i}) A_k,
$$
so that we may write
$$
w_{o,i}=\sum_{k\neq 0} \nabla \Phi_i^{-1} b_{i,k} e^{i\lambda_{q+1} k\cdot \Phi_i}.
$$
The following lemma ensures that the constant $M$ from Proposition \ref{p:iterativeprop} is geometric and, in particular, does not depend on all the parameters entering in the scheme.

\begin{lemma}\label{l:gemetric_const}
There exists a geometric constant $\overline C>0$ such that 
$$
\|b_{i,k}\|_0\leq \frac{\overline C}{|k|^5}\delta_{q+1}^{\sfrac12}.
$$
\end{lemma}
\begin{proof}
Apply \eqref{e:a_k_est} with $N=0$, $m=5$ and  $\mathcal{N}=\overline B_{\frac12}(\Id)$.
\end{proof}
 We are now ready to define the geometric constant $M$ of Proposition \ref{p:iterativeprop}.
 \begin{definition}\label{d:constant_M}
The constant $M$ is defined as 
$$
M=64 \overline C \sum_{k\in \mathbb{Z}^3\setminus \{0\}} \frac{1}{|k|^4},
$$
where $\overline C$ is the constant of Lemma \ref{l:gemetric_const}.
\end{definition}
To ensure that $w_{q+1}$ is divergence free we will add a corrector term $w_c$ to $w_o$. More precisely, in view of \eqref{e:Mikado_Potential}, we define
$$
w_c:=\frac{-i}{\lambda_{q+1}}\sum_{i,k\neq 0} \eta_p^i \delta_{q+1}^{\sfrac12}\nabla a_k(\tilde R_{q,i}) \times \frac{\nabla \Phi_i^T(k\times A_k)}{|k|^2} e^{i\lambda_{q+1} k\cdot \Phi_i}=\sum_{i,k\neq 0} c_{i,k} e^{i\lambda_{q+1} k\cdot \Phi_i},
$$
where 
$$
c_{i,k}:=\frac{-i}{\lambda_{q+1}} \eta_p^i \delta_{q+1}^{\sfrac12}\nabla a_k(\tilde R_{q,i}) \times \frac{\nabla \Phi_i^T(k\times A_k)}{|k|^2} .
$$
By using \eqref{e:Mikado_Potential}, one can check that 
$$
w_{q+1}:=w_o+w_c=\frac{-1}{\lambda_{q+1}} \curl \left(\sum_{i,k\neq 0} \nabla \Phi_i^T \frac{ik\times b_{i,k}}{|k|^2} e^{i\lambda_{q+1} k\cdot \Phi_i} \right),
$$
from which we deduce $\div w_{q+1}=0$. Finally, note that thanks to the cutoffs $\eta_p^i$ from Lemma \ref{l:cutoff_pert}, we also get 
\begin{equation}\label{support_pert}
\supp w_{q+1}\subset \T^3\times \left\{t\in \mathcal{B}_{q+1} \, : \, \dist(t,\mathcal{G}_{q+1})> \tau_{q+1} \right\} = \widehat{\mathcal{B}}_{q+1} \,.
\end{equation}
Recalling \eqref{e:pglueing1} and the inductive assumption \ref{Bq:v} on $v_q$, this guarantees the property \ref{Bq:v} at step $q+1$ . Moreover, by \eqref{e:pglueing6} and \eqref{support_pert} we also get \ref{Bq:vii} at step $q+1$, since for $N\geq 0$
 $$
 \|v_{q+1}(t)\|_{N+1}\leq \|\overline v_q(t)+w_{q+1}(t)\|_{N+1} =\|\overline v_q(t)\|_{N+1}\leq \delta_q^{\sfrac12}\lambda_q\ell_q^{-N}, \quad \forall t\in \widehat{\mathcal{B}}^c_{q+1}.
 $$
 \subsection{The final Reynolds stress and property \ref{Bq:vi}}
We define the new Reynolds stress as 
\begin{equation}\label{d:R_q+1}
\begin{split}
R_{q+1}&:=\mathcal{R}(w_{q+1}\cdot \nabla \overline v_q) + \mathcal{R}(\partial_t w_{q+1} +\overline{v}_q \cdot \nabla w_{q+1})+ \mathcal{R}\div (\overline R_q+ w_{q+1}\otimes w_{q+1})\\
&=R_{nash}+ R_{transp}+ R_{osc}.
\end{split}
\end{equation}
Notice that in all the three terms of the previous formula, the operator $\mathcal{R}$ is always applied to a divergence of a curl (thus to zero average vector fields). Moreover, by \eqref{support_pert} and \eqref{e:pglueing2}, we directly get  
$$
\supp R_{q+1}\subset \T^3\times \widehat{\mathcal{B}}_{q+1},
$$
which proves property \ref{Bq:vi} at step $q+1$.
With this definition, one may check that 
$$
\left\{\begin{array}{l}
\partial_t v_{q+1}+\div (v_{q+1}\otimes v_{q+1})  +\nabla p_{q+1} =\div R_{q+1}\\
\div v_{q+1} = 0,
\end{array}\right.
$$
where the new pressure is defined as $p_{q+1}:=\overline p_q$.

\subsection{Estimate on the perturbation}
We start by estimating all the terms entering in the definition of $w_{q+1}$. 
\begin{proposition}\label{p:ugly1}
For all $t\in \supp \eta_p^i$ and every $N\geq 0\,,$ we have 
\begin{align}
\norm{ (\nabla\Phi_i)^{-1}}_N + \norm{\nabla\Phi_i}_N &\lesssim \ell^{-N} \,,\label{e:phi_N}\\
\norm{\tilde R_{q,i}}_N &\lesssim  \ell^{-N}\,,\label{e:tR_est}\\
\norm{b_{i,k}}_N &\lesssim \delta_{q+1}^{\sfrac12}|k|^{-6}\ell^{-N}\,, \label{e:b_k_est_N}\\
\norm{c_{i,k}}_N &\lesssim  \delta_{q+1}^{\sfrac12}\lambda_{q+1}^{-1}|k|^{-6}\ell^{-N-1}\,.\label{e:c_k_est}
\end{align}
\end{proposition}
\begin{proof}
Let $t\in \supp \eta_p^i$. From \eqref{e:Dphi_near_id}, \eqref{e:Dphi_N},  \eqref{e:pglueing6} and \eqref{phi_close_id}, we obtain 
\begin{align*}
\|\nabla \Phi_i\|_N \lesssim \|\nabla \Phi_i\|_0+[\nabla \Phi_i]_N \lesssim 1+\|\nabla \Phi_i-\Id\|_0+[\nabla \Phi_i]_N\lesssim 1+\tau_{q+1}\|\nabla \overline v_q\|_N\lesssim \ell^{-N}.
\end{align*}
Moreover, by also using \eqref{e:pglueing7} we get 
$$
\|\tilde R_{q,i}\|_N\lesssim \| \nabla \Phi_i\|_N\| \nabla \Phi_i\|_0+\| \nabla \Phi_i\|_0^2\left\|\frac{\overline R_q}{\delta_{q+1}} \right\|_N\lesssim \ell^{-N}+\ell^{-N+\alpha}\lesssim \ell^{-N},
$$
which proves \eqref{e:tR_est}.
Finally, by the two previous estimates we also deduce 
$$
\|b_{i,k}\|_N\lesssim \delta_{q+1}^{\sfrac12} \|\tilde R_{q,i}\|_N|k|^{-6}\lesssim \delta_{q+1}^{\sfrac12} \ell^{-N}|k|^{-6}\,,
$$
where the constant in the inequality only depends on $N$ (see \eqref{e:a_k_est}). Similarly,
$$
\|c_{i,k}\|_N\lesssim \delta_{q+1}^{\sfrac12} \lambda_{q+1}^{-1}|k|^{-6}\left( \|\tilde R_{q,i}\|_{N+1}+\|\nabla \Phi_i\|_{N}\right)\lesssim \delta_{q+1}^{\sfrac12}\lambda_{q+1}^{-1}|k|^{-6}\ell^{-N-1}.
$$
\end{proof}

\begin{corollary}\label{c:pert_est}
If $a\gg 1$ is sufficiently large, the perturbation satisfies the estimates 
\begin{align}
\norm{w_o}_0 +\frac{1}{\lambda_{q+1}}\norm{w_o}_1 &\leq \frac{M}{4}\delta_{q+1}^{\sfrac 12},\label{e:w_o_est}\\
\norm{w_c}_0+\frac{1}{\lambda_{q+1}} \norm{w_c}_1 &\lesssim \delta_{q+1}^{\sfrac 12}\ell^{-1}\lambda_{q+1}^{-1},\label{e:w_c_est}\\
\norm{w_{q+1}}_0 +\frac{1}{\lambda_{q+1}}\norm{w_{q+1}}_1 &\leq  \frac{M}{2} \delta_{q+1}^{\sfrac 12}.\label{e:w_est}
\end{align}
In particular, \eqref{v_q+1-bar v_q} holds.
\end{corollary}
\begin{proof}
From \eqref{phi_close_id}, we deduce that $\|\nabla\Phi_i\|_0\leq 2$ on $\supp \eta_p^i$. Thus, since $\eta_p^i$ have disjoint supports, from Lemma \ref{l:gemetric_const} we get 
\begin{equation}\label{e:wo_0}
\|w_o\|_0\leq 2\delta_{q+1}^{\sfrac12} \overline{C}\sum_{k\neq 0} \frac{1}{|k|^5}\leq \frac{M}{32}\delta_{q+1}^{\sfrac12} .
\end{equation}
To estimate $\|w_o\|_1\,,$ we first observe that 
$$
\left\| \nabla\left( e^{i\lambda_{q+1}k\cdot \Phi_i}\right)\right\|_0\leq \lambda_{q+1}|k|\|\nabla \Phi_i\|_0\leq 2\lambda_{q+1}|k|.
$$
Compute now
$$
\nabla w_{o,i}=\sum_{k\neq 0} \nabla \Phi_i^{-1} b_{i,k}  \nabla\left( e^{i\lambda_{q+1}k\cdot \Phi_i}\right)+\sum_{k\neq 0} \nabla \left(\nabla \Phi_i^{-1} b_{i,k} \right)e^{i\lambda_{q+1}k\cdot \Phi_i}.
$$
In particular, from Lemma \ref{l:gemetric_const} and Proposition \ref{p:ugly1} we infer
$$
\|\nabla w_o\|_0\leq 4\delta_{q+1}^{\sfrac12}\lambda_{q+1} \overline C \sum_{k\neq 0} \frac{1}{|k|^4}+C\delta_{q+1}^{\sfrac12}\ell^{-1} \sum_{k\neq 0} \frac{1}{|k|^5}\leq \frac{M}{16}\delta_{q+1}^{\sfrac12}\lambda_{q+1} +C \delta_{q+1}^{\sfrac12} \ell^{-1},
$$
for some constant $C$ which also depends on $M$. Thanks to the parameter inequality $\ell^{-1}\ll \lambda_{q+1}$ from  \eqref{ell_less_lambda_q+1}, by choosing $a\gg 1 $ sufficiently large, we get 
$$
\|\nabla w_o\|_0\leq\frac{M}{8}\delta_{q+1}^{\sfrac12}\lambda_{q+1},
$$
which, together with \eqref{e:wo_0}, gives \eqref{e:w_o_est}. As a consequence of \eqref{e:c_k_est}, we also obtain \eqref{e:w_c_est}. Finally, estimate \eqref{e:w_est} follows by putting together \eqref{e:w_o_est} and \eqref{e:w_c_est} and using again $\ell^{-1}\ll \lambda_{q+1}\,.$
\end{proof}

We denote by $D_{t,q}:=\partial_t+\overline v_q \cdot \nabla $ the advective derivative with respect to $\overline v_q$. We have

\begin{proposition}\label{p:ugly2}
For $t\in \supp \eta_p^i$ and every $N\geq 0$ we have 
\begin{align}
\norm{ D_{t,q} \nabla\Phi_i}_N  &\lesssim\delta_q^{\sfrac12}\lambda_q \ell^{-N} \,,\label{e:Dtq_phi_N}\\
\norm{D_{t,q} \tilde R_{q,i}}_N &\lesssim \delta_q^{\sfrac12}\lambda_q^{1+\gamma}\ell^{-N-2\alpha}\,,\label{e:Dtq_tR_est}\\
\norm{D_{t,q} c_{i,k}}_N &\lesssim \delta_{q+1}^{\sfrac12}\delta_q^{\sfrac12}\lambda_q^{1+\gamma}\lambda_{q+1}^{-1}\ell^{-N-1-3\alpha}|k|^{-6}\,,\label{e:Dtq_c_k_est}\\
\norm{D_{t,q} b_{i,k}}_N &\lesssim \delta_{q+1}^{\sfrac12} \delta_q^{\sfrac12}\lambda_q^{1+\gamma}\ell^{-N-3\alpha}|k|^{-6}\,.\label{e:Dtq_b_k_est}
\end{align}
\end{proposition}
\begin{proof}Observe that $D_{t,q} \nabla \Phi_i=-\nabla \Phi_i D\overline v_q$. Thus, from \eqref{e:pglueing6} and \eqref{e:phi_N} we get
$$
\|D_{t,q}\nabla \Phi_i\|_N\lesssim \|\nabla\Phi_i\|_0\|\overline v_q\|_{N+1}+\|\nabla\Phi_i\|_N\|\overline v_q\|_1\lesssim \delta_q^{\sfrac12}\lambda_q\ell^{-N}.
$$
Differentiating \eqref{d:tildeRqi} yields
\begin{align*}
D_{t,q}\tilde R_{q,i}=D_{t,q}\nabla \Phi_i \left(\Id-\frac{\overline R_q}{\delta_{q+1}} \right) \nabla\Phi_i^T-\nabla \Phi_i\frac{D_{t,q}\overline R_q}{\delta_{q+1}}\nabla \Phi_i^T+\nabla \Phi_i\left(\Id-\frac{\overline R_q}{\delta_{q+1}} \right)D_{t,q} \nabla\Phi_i^T.
\end{align*}
Then, by \eqref{e:pglueing7}, \eqref{e:pglueing8},  \eqref{e:phi_N} and \eqref{e:Dtq_phi_N} we get
\begin{align*}
\left\|D_{t,q} \tilde R_{q,i}\right\|_N&\lesssim \|D_{t,q} \nabla \Phi_i\|_N +\left\| D_{t,q}\nabla\Phi_i\right\|_0\left(\frac{\left\| \overline R_q\right\|_N}{\delta_{q+1}}+\left\|\nabla\Phi_i^T\right\|_N \right)\\
&+\|\nabla \Phi_i\|_N \frac{\left\| D_{t, q} \overline R_q\right\|_0}{\delta_{q+1}}+\frac{\left\|D_{t,q} \overline R_q\right\|_N}{\delta_{q+1}}\\
&\lesssim \delta_q^{\sfrac12}\lambda_q\ell^{-N} +\delta_q^{\sfrac12}\lambda_q^{1+\gamma} \ell^{-N-2\alpha}\lesssim \delta_q^{\sfrac12}\lambda_q^{1+\gamma} \ell^{-N-2\alpha},
\end{align*}
which gives \eqref{e:Dtq_tR_est}. Compute now
\begin{align*}
D_{t,q}c_{i,k}&=-i\frac{\delta_{q+1}^{\sfrac12}}{\lambda_{q+1}}\left(\partial_t\eta_p^i \nabla \left(a_k\left(\tilde R_{q,i}\right)\right)+\eta_p^iD_{t,q}\nabla \left(a_k\left(\tilde R_{q,i}\right)\right) \right) \times \frac{\nabla\Phi_i^T(k\times A_k)}{|k|^2}\\
&-i\frac{\delta_{q+1}^{\sfrac12}}{\lambda_{q+1}}\eta_p^i\nabla \left(a_k\left(\tilde R_{q,i}\right)\right)\times  \frac{D_{t,q}\nabla\Phi_i^T(k\times A_k)}{|k|^2}.
\end{align*}
Writing $D_{t, q} \nabla \tilde R_{q,i} = \nabla (D_{t, q} \tilde R_{q,i}) - \nabla \overline v_q \nabla \tilde R_{q,i}$, we have by Proposition \ref{p:ugly1}, \eqref{est_cutoff_pert}, the previous two estimates \eqref{e:Dtq_phi_N} --\eqref{e:Dtq_tR_est} and \eqref{e:pglueing6} that 
\begin{align*}
\frac{|k|^6\lambda_{q+1}}{\delta_{q+1}^{\sfrac12}}\|D_{t,q} c_{i,k}\|_N&\lesssim \norm{\partial_t\eta_p^i}_0\norm{ \tilde R_{q,i}}_{N+1}+\norm{D_{t,q} \nabla \tilde R_{q,i}}_{N}\\
&+\left(\norm{\partial_t\eta_p^i}_0 \norm{\tilde R_{q,i}}_1+\norm{D_{t,q} \nabla \tilde R_{q,i}}_{0} \right) \norm{\nabla \Phi_i^T}_N\\
& +\norm{\tilde R_{q,i}}_{N+1}\norm{D_{t,q} \nabla \Phi_i^T}_0+\norm{\tilde R_{q,i}}_{1}\norm{D_{t,q} \nabla \Phi_i^T}_N\\
&\lesssim \tau_{q+1}^{-1}\ell^{-N-1}+\delta_q^{\sfrac12}\lambda_q^{1+\gamma}\ell^{-N-1-2\alpha}\\
&+\left( \tau_{q+1}^{-1}\ell^{-1}+\delta_q^{\sfrac12}\lambda_q^{1+\gamma}\ell^{-1-2\alpha}\right)\ell^{-N}+\ell^{-N-1}\delta_q^{\sfrac12}\lambda_q\\
&\lesssim  \tau_{q+1}^{-1}\ell^{-N-1}+\delta_q^{\sfrac12}\lambda_q^{1+\gamma}\ell^{-N-1-2\alpha}\\
&=\delta_q^{\sfrac12}\lambda_q^{1+\gamma}\left( \lambda_q^{3\alpha}+\ell^{-2\alpha}\right) \ell^{-N-1}\lesssim \delta_q^{\sfrac12}\lambda_q^{1+\gamma}\ell^{-N-1-3\alpha},
\end{align*}
where in the last inequality we also used that $\lambda_q^{3\alpha}\leq \ell^{-3\alpha}$ from \eqref{l:bound}.

With similar computations we also get \eqref{e:Dtq_b_k_est}. Indeed
\begin{align*}
\norm{D_{t,q} b_{i,k}}_N&\lesssim\frac{\delta_{q+1}^{\sfrac12}}{|k|^6}\left(\norm{\partial_t \eta_p^i}_0\norm{\tilde R_{q,i}}_N+\norm{D_{t,q}\tilde R_{q,i}}_N \right)\\
&\lesssim \frac{\delta_{q+1}^{\sfrac12}}{|k|^6}\left( \tau_{q+1}^{-1}\ell^{-N}+\delta_q^{\sfrac12}\lambda_q^{1+\gamma} \ell^{-N-2\alpha}\right)\lesssim\delta_{q+1}^{\sfrac12}\delta_q^{\sfrac12}\lambda_q^{1+\gamma} \ell^{-N-3\alpha}|k|^{-6}.
\end{align*}
\end{proof}

\subsection{Estimate  on the new Reynolds stress} In this final section, we prove our last estimate \eqref{R_q+1} in order to conclude the proof of Proposition \ref{p:perturbation}.
\begin{proposition}\label{p:finalR}
The Reynolds stresses $R_{nash}$, $R_{osc}$ and $R_{transp}$ defined in \eqref{d:R_q+1} satisfy
\begin{align}
\|R_{nash}\|_\alpha&\lesssim \frac{\delta_{q+1}^{\sfrac12}\delta_q^{\sfrac12}\lambda_q}{\lambda_{q+1}^{1-2\alpha}},\label{e:Rnash} \\
\|R_{osc}\|_\alpha&\lesssim \frac{\delta_{q+1}^{\sfrac12}\delta_q^{\sfrac12}\lambda_q^{1+\sfrac{\gamma}{2}}}{\lambda_{q+1}^{1-4\alpha}},\label{e:Rosc} \\
\|R_{transp}\|_\alpha & \lesssim  \frac{\delta_{q+1}^{\sfrac12}\delta_q^{\sfrac12}\lambda_q^{1+\gamma}}{\lambda_{q+1}^{1-5\alpha}},  \label{e:Rtransp.}
\end{align}
for $\alpha >0$ sufficiently small and $a\gg 1$ large enough. In particular, \eqref{R_q+1} holds. 
\end{proposition}
\begin{proof}[Proof of \eqref{e:Rnash}]
We rewrite the term $R_{nash}$ as 
$$
\mathcal R (w_{q+1}\cdot \nabla v_q)=\sum_{i,k\neq 0} \mathcal R \left(\left(\nabla \Phi_i^{-1}b_{i,k} e^{i\lambda_{q+1}k\cdot \Phi_i} + c_{i,k}e^{i\lambda_{q+1}k \cdot \Phi_i}\right) \cdot \nabla \overline v_q\right).
$$
Using Proposition \ref{p:R:oscillatory_phase} and Proposition \ref{p:ugly1}, we estimate
\begin{align*}
&\left\| \mathcal R\left(\nabla \Phi_i^{-1}b_{i,k} e^{i\lambda_{q+1}k\cdot \Phi_i}  \cdot \nabla \overline v_q\right) \right\|_\alpha \\
&\lesssim \frac{\|\nabla\Phi_i^{-1} b_{i,k} \cdot \nabla \overline v_q\|_0}{\lambda_{q+1}^{1-\alpha}}+\frac{\|\nabla\Phi_i^{-1} b_{i,k} \cdot \nabla \overline  v_q\|_{N+\alpha}+\|\nabla\Phi_i^{-1} b_{i,k} \cdot \nabla \overline v_q\|_0\|\Phi_i\|_{N+\alpha}}{\lambda_{q+1}^{N-\alpha}}\\
&\lesssim  \frac{\delta_{q+1}^{\sfrac12}\delta_q^{\sfrac12}\lambda_q}{\lambda_{q+1}^{1-\alpha}|k|^6}+ \frac{\delta_{q+1}^{\sfrac12}\delta_q^{\sfrac12}\lambda_q}{\lambda_{q+1}^{N-\alpha}\ell^{N+\alpha}|k|^6}\lesssim  \frac{\delta_{q+1}^{\sfrac12}\delta_q^{\sfrac12}\lambda_q}{\lambda_{q+1}^{1-2\alpha}|k|^6}\left(1+\frac{\lambda_{q+1}}{\left(\lambda_{q+1}\ell\right)^N} \right),
\end{align*}
where in the last inequality we also used that $\ell^{-\alpha}\leq \lambda_{q+1}^{\alpha}$ by \eqref{ell_less_lambda_q+1}. We claim that, by choosing $N\gg 1$ sufficiently large (depending on $\beta, \gamma, \alpha, b$), then it holds
\begin{equation}\label{param_ineq_N}
\frac{\lambda_{q+1}}{\left(\lambda_{q+1}\ell\right)^N}\leq 1,
\end{equation}
for $a\gg 1$ sufficiently large. Indeed we have 
$$
\frac{\lambda_{q+1}}{\left(\lambda_{q+1}\ell\right)^N}\lesssim a^{b-N(b+\beta-1-\gamma-3\alpha-\beta b)}.
$$
Thus, there exists $N$ large enough such that \eqref{param_ineq_N} holds, if $b+\beta-1-\gamma-3\alpha-\beta b>0$, which is equivalent to
\begin{equation}\label{param_relat_nash}
\gamma+3\alpha<b-1-\beta(b-1)=(b-1)(1-\beta).
\end{equation}
Since the right hand side in \eqref{param_relat_nash} is strictly larger than the upper bound on $\gamma$ in \eqref{gamma:bound2}, we conclude that \eqref{param_relat_nash} (and so \eqref{param_ineq_N}) holds if $\alpha>0$ is sufficiently small. Hence we achieved
$$
\left\| \mathcal R\left(\nabla \Phi_i^{-1}b_{i,k} e^{i\lambda_{q+1}k\cdot \Phi_i}  \cdot \nabla \overline v_q\right) \right\|_\alpha\lesssim  \frac{\delta_{q+1}^{\sfrac12}\delta_q^{\sfrac12}\lambda_q}{\lambda_{q+1}^{1-2\alpha}|k|^6}.
$$
Due to the fact that the estimates on the coefficients $c_{i,k}$ from Proposition \ref{p:ugly1} are better than the ones on the $b_{i,k}$ by \eqref{ell_less_lambda_q+1}, we also get that
$$
\left\| \mathcal R\left(c_{i,k} e^{i\lambda_{q+1}k\cdot \Phi_i}  \cdot \nabla \overline v_q\right) \right\|_\alpha\lesssim  \frac{\delta_{q+1}^{\sfrac12}\delta_q^{\sfrac12}\lambda_q}{\lambda_{q+1}^{1-2\alpha}|k|^6}.
$$
Finally, summing over all the frequencies $k\neq 0$, we conclude the desired estimate \eqref{e:Rnash}.
\end{proof}

\begin{proof}[Proof of \eqref{e:Rosc}]
We write the oscillation error as follows
\begin{align*}
R_{osc}&=\mathcal{R} \div (\overline R_q + w_{q+1}\otimes w_{q+1})
\\
&=\mathcal R \div (\overline R_q -\delta_{q+1}\Id + w_o\otimes w_o)+ \mathcal{R}\div (w_o\otimes w_c +w_c\otimes w_o+w_c\otimes w_c)\\
&=: \mathcal O_1+\mathcal O_2.
\end{align*}
Since, by Schauder estimates, the operator $\mathcal{R}\div :C^\alpha(\T^3)\rightarrow C^\alpha(\T^3)$ is bounded, we deduce by using \eqref{e:w_o_est},  \eqref{e:w_c_est} and \eqref{ell_less_lambda_q+1}
\begin{align*}
\|\mathcal O_2\|_\alpha&\lesssim \|w_o\|_\alpha \|w_c\|_\alpha+ \|w_c\|^2_\alpha\lesssim \frac{\delta_{q+1}\lambda_{q+1}^{2\alpha}}{\lambda_{q+1}\ell}+\frac{\delta_{q+1}\lambda_{q+1}^{2\alpha}}{(\lambda_{q+1}\ell)^2}\\
&\lesssim   \frac{\delta_{q+1}\lambda_{q+1}^{2\alpha}}{\lambda_{q+1}\ell}=\frac{\delta_{q+1}^{\sfrac12}\delta_q^{\sfrac12}\lambda_q^{1+\sfrac{\gamma}{2}+\sfrac{3\alpha}{2}}}{\lambda_{q+1}^{1-2\alpha}}\leq \frac{\delta_{q+1}^{\sfrac12}\delta_q^{\sfrac12}\lambda_q^{1+\sfrac{\gamma}{2}}}{\lambda_{q+1}^{1-4\alpha}}.
\end{align*}
Thus, to conclude the desired estimate on $R_{osc}$ we are only left with $\mathcal{O}_1$. Since by Lemma \ref{l:cutoff_pert} the supports of $\eta_p^i$ are disjoint and $ \sum_i \left(\eta_p^i\right)^2\equiv 1$ on the $\supp \overline R_q$, we have
$$
\mathcal{O}_1=\mathcal R \div  \sum_i  \left( \left(\eta_p^i\right)^2\left(\overline R_q-\delta_{q+1}\Id\right) +w_{o,i}\otimes w_{o,i}\right).
$$
Using \eqref{e:Mikado_stationarity}, we can write
\begin{align*}
w_{o,i}\otimes w_{o,i}&= \delta_{q+1}\left(\eta_p^i\right)^2 \nabla \Phi_i^{-1}\left( W\otimes W\right) (\tilde R_{q,i},\lambda_{q+1}\Phi_i)\nabla \Phi_i^{-T}\\
&=\delta_{q+1}\left(\eta_p^i\right)^2 \nabla \Phi_i^{-1}\tilde R_{q,i}\nabla \Phi_i^{-T}+\sum_{k\neq 0}\delta_{q+1}\left(\eta_p^i\right)^2 \nabla \Phi_i^{-1} C_k(\tilde R_{q,i})  \nabla \Phi_i^{-T} e^{i\lambda_{q+1}k\cdot \Phi_i}\\
&=\left(\eta_p^i\right)^2(\delta_{q+1}\Id-\overline R_q)+ \sum_{k\neq 0}\delta_{q+1}\left(\eta_p^i\right)^2 \nabla \Phi_i^{-1} C_k(\tilde R_{q,i})  \nabla \Phi_i^{-T} e^{i\lambda_{q+1}k\cdot \Phi_i},
\end{align*}
from which we deduce 
$$
\mathcal{O}_1=\mathcal R \div \left( \sum_{i,k\neq 0}\delta_{q+1}\left(\eta_p^i\right)^2 \nabla \Phi_i^{-1} C_k(\tilde R_{q,i})  \nabla \Phi_i^{-T} e^{i\lambda_{q+1}k\cdot \Phi_i} \right).
$$
Also, recalling \eqref{e:Ck_ind} 
$$
\nabla \Phi_i^{-1} C_k \nabla \Phi_i^{-T}\nabla \Phi_i^T k=0,
$$
and consequently 
$$
\mathcal{O}_1=\mathcal R \left( \sum_{i,k\neq 0}\delta_{q+1}\left(\eta_p^i\right)^2  \div\left( \nabla \Phi_i^{-1} C_k(\tilde R_{q,i})  \nabla \Phi_i^{-T}\right) e^{i\lambda_{q+1}k\cdot \Phi_i} \right).
$$
Thus, again by Proposition \ref{p:R:oscillatory_phase} and Proposition \ref{p:ugly1},  we estimate on $\supp \eta_p^i$
\begin{align*}
&\left\| \mathcal R \left(  \div\left( \nabla \Phi_i^{-1} C_k(\tilde R_{q,i})  \nabla \Phi_i^{-T}\right) e^{i\lambda_{q+1}k\cdot \Phi_i} \right)\right\|_\alpha \lesssim \frac{\left\|  \div\left( \nabla \Phi_i^{-1} C_k(\tilde R_{q,i})  \nabla \Phi_i^{-T}\right)  \right\|_0}{\lambda_{q+1}^{1-\alpha}}\\
&+\frac{    \left\|  \div\left( \nabla \Phi_i^{-1} C_k(\tilde R_{q,i})  \nabla \Phi_i^{-T}\right)  \right\|_{N+\alpha}+\left\| \div\left( \nabla \Phi_i^{-1} C_k(\tilde R_{q,i})  \nabla \Phi_i^{-T}\right)  \right\|_0\|\Phi_i\|_{N+\alpha}               }{\lambda_{q+1}^{N-\alpha}}\\
&\lesssim  \frac{1}{\ell^{1+\alpha}\lambda_{q+1}^{1-\alpha}|k|^6} + \frac{1}{\lambda_{q+1}^{N-\alpha} \ell^{N+1+\alpha} \lvert k \rvert^6 } \lesssim \frac{1}{\ell\lambda_{q+1}^{1-2\alpha}|k|^6},
\end{align*}
where we have again chosen $N\gg1$ large enough to get the desired estimate, together with $\ell^{-\alpha}\leq \lambda_{q+1}^\alpha$ (see \eqref{ell_less_lambda_q+1}), for the last inequality. By summing over $k\neq 0$ we conclude, reusing \eqref{ell_less_lambda_q+1}, that
$$
\|\mathcal O_1\|_\alpha\lesssim \frac{\delta_{q+1}}{\ell\lambda_{q+1}^{1-2\alpha}}\lesssim \frac{\delta_{q+1}^{\sfrac12}\delta_q^{\sfrac12}\lambda_q^{1+\sfrac{\gamma}{2}}}{\lambda_{q+1}^{1-4\alpha}}.
$$
\end{proof}
\begin{proof}[Proof of \eqref{e:Rtransp.}] We start by splitting the transport error into two parts
$$
\mathcal R \left(\left( \partial_t+\overline v_q\cdot \nabla \right)w_{q+1}\right)=\mathcal R \left(\left( \partial_t+\overline v_q\cdot \nabla \right)w_{o}\right)+\mathcal R \left(\left( \partial_t+\overline v_q\cdot \nabla \right)w_{c}\right)=:\mathcal T_1+\mathcal T_2.
$$
We start with $\mathcal T_1$. Applying \eqref{rel_div_free} yields
\begin{align}\label{split_transport_principal}
\left( \partial_t+\overline v_q\cdot \nabla \right)w_{o}=\sum_{i,k\neq 0} \left(\nabla \overline v_q \right)^T \nabla \Phi_i^{-1}b_{i,k} e^{i\lambda_{q+1} k\cdot \Phi_i}+\sum_{i,k\neq 0}\nabla \Phi_i^{-1}D_{t,q}b_{i,k} e^{i\lambda_{q+1} k\cdot \Phi_i}.
\end{align}
We now apply Proposition \ref{p:R:oscillatory_phase}, together with Proposition \ref{p:ugly1}, to obtain 
\begin{equation}\label{est_1_principal}
\begin{split}
&\norm{\mathcal{R}\left(\left( \nabla \overline v_q\right)^T \nabla \Phi_i^{-1} b_{i,k} e^{i\lambda_{q+1}k\cdot \Phi_i} \right)}_\alpha\\
&\lesssim \frac{\left\| \left( \nabla \overline v_q\right)^T \nabla \Phi_i^{-1} b_{i,k}  \right\|_0}{\lambda_{q+1}^{1-\alpha}}+\frac{\left\| \left( \nabla \overline v_q\right)^T \nabla \Phi_i^{-1} b_{i,k}  \right\|_{N+\alpha}+\left\| \left( \nabla \overline v_q\right)^T \nabla \Phi_i^{-1} b_{i,k}  \right\|_0\|\Phi_i\|_{N+\alpha}}{\lambda_{q+1}^{N-\alpha}}\\
&\lesssim \frac{\delta_{q+1}^{\sfrac12}\delta_q^{\sfrac12}\lambda_q}{\lambda_{q+1}^{1-\alpha}|k|^6}+ \frac{\delta_{q+1}^{\sfrac12}\delta_q^{\sfrac12}\lambda_q}{\lambda_{q+1}^{N-\alpha}\ell^{N+\alpha}|k|^6}\lesssim  \frac{\delta_{q+1}^{\sfrac12}\delta_q^{\sfrac12}\lambda_q}{\lambda_{q+1}^{1-2\alpha}|k|^6},
\end{split}
\end{equation}
where in the last inequality we have again chosen $\alpha>0$ sufficiently small and $N\gg 1$ large enough.

The estimate on the second term in \eqref{split_transport_principal} follows by Proposition \ref{p:R:oscillatory_phase} and Proposition \ref{p:ugly2}. Indeed we have
\begin{equation}\label{est_2_principal}
\begin{split}
&\norm{\mathcal{R}\left(\nabla \Phi_i^{-1}D_{t,q}b_{i,k} e^{i\lambda_{q+1} k\cdot \Phi_i} \right)}_\alpha\\
&\lesssim \frac{\left\|\nabla \Phi_i^{-1}D_{t,q}b_{i,k}  \right\|_0}{\lambda_{q+1}^{1-\alpha}}+\frac{\left\| \nabla \Phi_i^{-1}D_{t,q}b_{i,k} \right\|_{N+\alpha}+\left\| \nabla \Phi_i^{-1}D_{t,q}b_{i,k} \right\|_0\|\Phi_i\|_{N+\alpha}}{\lambda_{q+1}^{N-\alpha}}\\
&\lesssim \frac{\delta_{q+1}^{\sfrac12}\delta_q^{\sfrac12}\lambda_q^{1+\gamma}\ell^{-3\alpha}}{\lambda_{q+1}^{1-\alpha}|k|^6}+ \frac{\delta_{q+1}^{\sfrac12}\delta_q^{\sfrac12}\lambda_q^{1+\gamma}\ell^{-3\alpha}+\delta_{q+1}^{\sfrac12}\delta_q^{\sfrac12}\lambda_q^{1+\gamma}\ell^{1-3\alpha}}{\lambda_{q+1}^{N-\alpha}\ell^{N+\alpha}|k|^6}\\
&\lesssim  \frac{\delta_{q+1}^{\sfrac12}\delta_q^{\sfrac12}\lambda_q^{1+\gamma}\ell^{-4\alpha}}{\lambda_{q+1}^{1-\alpha}|k|^6}\lesssim  \frac{\delta_{q+1}^{\sfrac12}\delta_q^{\sfrac12}\lambda_q^{1+\gamma}}{\lambda_{q+1}^{1-5\alpha}|k|^6},
\end{split}
\end{equation}
where we also used $\ell^{-4\alpha}\leq \lambda_{q+1}^{4\alpha}$. Putting together \eqref{est_1_principal} and \eqref{est_2_principal}, summing over all the frequencies $k\neq 0$, we conclude 
\begin{equation}\label{est_T1}
\|\mathcal T_1\|_{\alpha}\lesssim  \frac{\delta_{q+1}^{\sfrac12}\delta_q^{\sfrac12}\lambda_q^{1+\gamma}}{\lambda_{q+1}^{1-5\alpha}}.
\end{equation}
To estimate $\mathcal T_2$ we observe that since $(\partial_t + \overline{v}_q \cdot \nabla) \Phi_i=0$ by choice of $\Phi_i$, we have
$$
\left( \partial_t+\overline v_q \cdot \nabla\right) w_c=\sum_{i,k\neq 0} \left( D_{t,q} c_{i,k}\right)e^{i\lambda_{q+1}k\cdot \Phi_i}.
$$
Then, applying once again Proposition \ref{p:R:oscillatory_phase}, we get
$$
\left\| \mathcal R\left( \left( D_{t,q} c_{i,k}\right) e^{i\lambda_{q+1}k\cdot \Phi_i}\right)\right\|_\alpha\lesssim  \frac{\left\|D_{t,q} c_{i,k} \right\|_0}{\lambda_{q+1}^{1-\alpha}}+\frac{\left\| D_{t,q} c_{i,k} \right\|_{N+\alpha}+\left\|D_{t,q} c_{i,k} \right\|_0\|\Phi_i\|_{N+\alpha}}{\lambda_{q+1}^{N-\alpha}}.
$$
Since from Proposition \ref{p:ugly2} the estimates on $D_{t,q} c_{i,k}$ are better than the ones for $D_{t,q} b_{i,k}$ (recall that $\ell^{-1}\leq \lambda_{q+1}$) we obtain, as for the estimate \eqref{est_2_principal}, that 
$$
\left\| \mathcal R\left( \left( D_{t,q} c_{i,k}\right) e^{i\lambda_{q+1}k\cdot \Phi_i}\right)\right\|_\alpha\lesssim  \frac{\delta_{q+1}^{\sfrac12}\delta_q^{\sfrac12}\lambda_q^{1+\gamma}}{\lambda_{q+1}^{1-5\alpha}|k|^6},
$$
from which, by summing over $k\neq 0$, we deduce 
\begin{equation}\label{est_T2}
\|\mathcal T_2\|_{\alpha}\lesssim  \frac{\delta_{q+1}^{\sfrac12}\delta_q^{\sfrac12}\lambda_q^{1+\gamma}}{\lambda_{q+1}^{1-5\alpha}}.
\end{equation}
Estimates \eqref{est_T1} and \eqref{est_T2} imply the validity of \eqref{e:Rtransp.} and this concludes the proof of Proposition \ref{p:finalR}.
\end{proof}

\appendix

%%%%%%%%%%%%%%%%%%%%%%%%%
\section{H\"older spaces}\label{s:hoelder}
%%%%%%%%%%%%%%%%%%%%%%%%%

In the following $m=0,1,2,\dots$, $\alpha\in (0,1)$, and $\theta$ is a multi-index. We introduce the usual (spatial) 
H\"older norms as follows.
First of all, the supremum norm is denoted by $\|f\|_0=\sup_{\T^3\times [0,T]}|f|$. We define the H\"older seminorms 
as
\begin{equation*}
\begin{split}
[f]_{m}&=\max_{|\theta|=m}\|D^{\theta}f\|_0\, ,\\
[f]_{m+\alpha} &= \max_{|\theta|=m}\sup_{x\neq y, t}\frac{|D^{\theta}f(x, t)-D^{\theta}f(y, t)|}{|x-y|^{\alpha}}\, ,
\end{split}
\end{equation*}
where $D^\theta$ are  space derivatives only.
The H\"older norms are then given by
\begin{eqnarray*}
\|f\|_{m}&=&\sum_{j=0}^m[f]_j,\\
\|f\|_{m+\alpha}&=&\|f\|_m+[f]_{m+\alpha}.
\end{eqnarray*}
Moreover, we write $[f (t)]_\alpha$ and $\|f (t)\|_\alpha$ when the time $t$ is fixed and the
norms are computed for the restriction of $f$ to the $t$-time slice.

Recall the following elementary inequalities
\begin{proposition}\label{p:interpolation}
Let $f, g$ be two smooth functions. For any $r\geq s\geq 0$ we have
\begin{align}
[fg]_{r}&\leq C\bigl([f]_r\|g\|_0+\|f\|_0[g]_r\bigr),\label{e:Holderproduct}\\
[f]_{s}&\leq C\|f\|_0^{1-\sfrac{s}{r}}[f]_{r}^{\sfrac{s}{r}}.\label{e:Holderinterpolation2}
\end{align}
\end{proposition}

We also recall the quadratic commutator estimate of~\cite{CET94}.

\begin{proposition}\label{p:CET}
Let $f,g\in C^{\infty}\left(\T^3\right)$ and $\psi$ a standard radial smooth and compactly supported kernel. For any $r\geq 0$  we have the estimate
\[
\Bigl\|(f*\psi_\ell)( g*\psi_\ell)-(fg)*\psi_\ell\Bigr\|_r\leq C\ell^{2-r}  \|f\|_1\|g\|_1 \, ,
\]
where the constant $C$ depends only on $r$.
\end{proposition}

We recall that, if the mollification kernel is radial, it also holds
\begin{equation}\label{improved_moll}
\| f-f_\ell\|_r\lesssim \ell^{2-r}\|f\|_2.
\end{equation}
\section{Local smooth solutions}

In this section we recall that from any smooth initial datum $\overline {v}$ there exists a smooth solution $v$ of \eqref{E}, where its maximal time of existence is proportional to $\|\overline {v}\|_{1+\alpha}^{-1}$. Indeed we have the following

\begin{proposition}\label{p:local:Euler}
For any $0<\alpha<1$ there exists a constant $c=c(\alpha)>0$ with the following property. Given any initial data $\overline  v \in C^{\infty}\left(\T^3\right)$ and $T\leq c\norm{\overline  v}_{1+\alpha}^{-1}$, there exists a unique smooth solution $v:\mathbb T^3 \times [-T,T]\rightarrow \mathbb R^3$  of \eqref{E} such that $v(0,\cdot)=\overline  v$.
Moreover, $v$ obeys the bounds
\begin{align}
\norm{v}_{N+\alpha} \lesssim &\norm{u_0}_{N+\alpha}~. \label{e:euler_eq_bd_k}
\end{align}
for all $N\geq 1$, where the implicit constant depends on $N$ and $\alpha>0$.
\end{proposition}

We refer to \cite[Propositon 3.1]{BDLSV2019} for the proof.

%%%%%%%%%%%%%%%%%%%%%%%%%
\section{Potential theory estimates}
%%%%%%%%%%%%%%%%%%%%%%%%%
We recall the definition of the standard class of periodic Calder{\'o}n-Zygmund operators.
Let $K$ be an $\R^3$ kernel which obeys the properties
\begin{itemize}
\item $ K(z) = w\left(\frac{z}{|z|}\right) |z|^{-3} $, for all $z\in\R^3 \setminus \{0\}$ 
\item $w \in C^\infty({\mathbb S}^2)$
\item $\int_{|\widehat z|=1} w(\widehat z) d\widehat z = 0$.
\end{itemize}
From the $\R^3$ kernel $K$, use Poisson summation to define the periodic kernel 
\begin{align*}
K_{\T^3}(z) =  K(z) + \sum_{\ell \in {\mathbb Z}^3 \setminus \{0\}} \left( K(z+\ell) - K(\ell) \right).
\end{align*}
Then the operator
\begin{align*}
T_K f(x) = p.v. \int_{\T^3} K_{\T^3}(x-y) f(y) dy
\end{align*}
is a $\T^3$-periodic Calder{\'o}n-Zygmund operator, acting on $\T^3$-periodic functions $f$ with zero mean on $\T^3$.
The following proposition, proving the boundedness of periodic Calder{\'o}n-Zygmund operators on periodic H\"older spaces is classical
\begin{proposition}
\label{p:CZO_C_alpha}
Fix $\alpha \in (0,1)$. Periodic Calder{\'o}n-Zygmund operators are bounded on the space of zero mean $\T^3$-periodic $C^\alpha$ functions. 
\end{proposition}

The following proposition is taken from \cite{BDLSV2019}.
\begin{proposition}\label{p:commutator} Let $\alpha\in (0,1)$ and $N\geq 0$. Let $T_K$ be a Calder{\'o}n-Zygmund operator with kernel $K$. Let $b\in C^{N+1,\alpha}\left(\T^3\right)$ be a divergence free vector field. Then we have
$$
\norm{\left[T_K, b\cdot \nabla \right] f}_{N+\alpha}\lesssim \|b\|_{1+\alpha}\|f\|_{N+\alpha}+\|b\|_{N+1+\alpha}\|f\|_{\alpha},
$$
for any $f\in C^{N+\alpha}\left( \T^3\right)$, where the implicit constant depends on $\alpha, N$ and $K$.
\end{proposition}

 \section{Stationary phase lemma}
The following is a simple consequence of classical stationary phase techniques.
\begin{proposition}
\label{p:R:oscillatory_phase}
Let  $\alpha \in(0,1)$ and $N \geq 1$. Let $a \in C^\infty\left(\T^3\right)$, $\Phi\in C^\infty\left(\T^3;\R^3\right)$ be smooth functions and assume that
\begin{equation*}
\widehat{C}^{-1}\leq \abs{\nabla \Phi}, | \nabla \Phi^{-1}|\leq \widehat{C}
\end{equation*}
holds on $\T^3$. Then for the operator $\RR$ defined in \eqref{e:R:def}, we have
 \begin{align*}
\norm{ {\mathcal R} \left(a(x) e^{i k\cdot \Phi} \right)}_{\alpha} 
&\lesssim \frac{  \norm{a}_{0}}{|k|^{1-\alpha}} +   \frac{ \norm{a}_{N+\alpha}+\norm{a}_{0}\norm{\Phi}_{N+\alpha}}{|k|^{N-\alpha}} \, ,
\end{align*}
where the implicit constant depends on $\widehat{C}$, $\alpha$ and $N$, but not on $k$.
\end{proposition}

%%%%%%%%%%%%%%%%%%%%%%%%%
\section{Estimates on the transport equation}
%%%%%%%%%%%%%%%%%%%%%%%%%

In this section we recall some well known results regarding smooth solutions of
the transport equation
\begin{equation}\label{e:transport}
\left\{\begin{array}{l}
\partial_t f + v\cdot  \nabla f =g\\ 
f|_{t_0}=f_0,
\end{array}\right.
\end{equation}
where $v=v(t,x)$ is a given smooth vector field. We will consider solutions
on the entire space $\R^3$ and treat solutions on the torus simply as periodic solution in $\R^3$.

\begin{proposition}\label{p:transport}
Assume $|t-t_0|\|v\|_1\leq 1$. Any solution $f$ of \eqref{e:transport} satisfies
\begin{align}
\|f(t)\|_\alpha &\lesssim \|f_0\|_\alpha + \int_{t_0}^t \|g (\tau)\|_\alpha\, d\tau\,,\label{e:trans_est_0}
\end{align}
for all $0\leq \alpha\leq 1$, and, more generally, for any $N\geq 1$ and $0\leq\alpha\leq 1$
\begin{align}
[f (t)]_N \leq [ f_0]_{N+\alpha} + (t-t_0)[v]_{N+\alpha} [f_0]_1 +\int_{t_0}^t \left([g (\tau)]_N + (t-\tau ) [v ]_N [g (\tau)]_{1}\right)\,d\tau.
\label{e:trans_est_1}
\end{align}
Define $\Phi (t, \cdot)$ to be the inverse of the flux $X$ of $v$ starting at time $t_0$ as the identity
(i.e. $\frac{d}{dt} X = v (X,t)$ and $X (x, t_0 )=x$). Under the same assumptions as above, it holds
\begin{align}
\norm{\nabla \Phi (t) -\Id}_0&\lesssim |t-t_0| [v]_1\,,  \label{e:Dphi_near_id}\\
[\Phi (t)]_N &\lesssim  |t-t_0| [v]_N\quad \forall N \geq 2.\label{e:Dphi_N}
\end{align}
\end{proposition}

%%%%%%%%%%%%% BIBLIOGRAPHY %%%%%%%%%%

\end{document}